\documentclass[11pt,a4paper]{article}
\usepackage{amsmath,amssymb,amsthm}
\usepackage{ascmac}
\usepackage{color}
\usepackage[utf8]{inputenc}
\theoremstyle{definition}
 \newtheorem{dfn}{Definition}
 \newtheorem{remark}[dfn]{Remark}

\theoremstyle{plain}
 \newtheorem{thm}[dfn]{Theorem}

\newtheorem{assump}[dfn]{Assumption}
 
\newcommand{\bk}{{\bold k}}
\newcommand{\bu}{{\bold u}}

\newcommand{\bv}{{\bold v}}
\newcommand{\bw}{{\bold w}}

\newcommand{\bff}{{\bold f}}

\newcommand{\bD}{{\bold D}}

\newcommand{\bG}{{\bold G}}

\newcommand{\bI}{{\bold I}}

\newcommand{\bV}{{\bold V}}

\newcommand{\DV}{{\rm Div}\,}
\newcommand{\dv}{{\rm div}\,}

\newcommand{\BR}{{\Bbb R}}
\newcommand{\BC}{{\Bbb C}}

\newcommand{\BN}{{\Bbb N}}

\newcommand{\CD}{{\mathcal D}}
\newcommand{\CE}{{\mathcal E}}
\newcommand{\CF}{{\mathcal F}}
\newcommand{\CI}{{\mathcal I}}

\newcommand{\CL}{{\mathcal L}}
\newcommand{\CM}{{\mathcal M}}

\newcommand{\CR}{{\mathcal R}}
\newcommand{\CS}{{\mathcal S}}
\newcommand{\CT}{{\mathcal T}}

\newcommand{\CP}{{\mathcal P}}

\newcommand{\CU}{{\mathcal U}}
\newcommand{\CV}{{\mathcal V}}

\newcommand{\fp}{{\frak p}}

\newcommand{\bg}{{\bold g}}

\newcommand{\pd}{\partial}

\begin{document}
\title{\bf New thought on Matsumura-Nishida theory in the $L_p$-$L_q$ maximal
regularity framework}
\author{
Yoshihiro Shibata  
\thanks{Department of Mathematics, 
Waseda University, 
Ohkubo 3-4-1, Shinjuku-ku, Tokyo 169-8555, Japan. \endgraf
e-mail address: yshibata@waseda.jp \endgraf
Adjunct faculty member in the Department of Mechanical Engineering and 
Materials Scinece, University of Pittsburgh \endgraf
partially supported by Top Global University Project and JSPS Grant-in-aid
for Scientific Research (A) 17H0109.
}
}
\date{}
\maketitle
\begin{abstract}
This paper is devoted to proving the global well-posedness of  
initial-boundary value problem for Navier-Stokes
 equations describing the motion of viscous, compressible, barotropic
fluid flows in a three dimensional exterior domain with non-slip
boundary conditions. This was first proved by an excellent paper
due to Matsumura and Nishida \cite{MN2} in 1983.  
In \cite{MN2}, they used energy method and 
their requirement was that  space derivatives of the mass density up to third order
and   space derivatives of the velocity fields up to fourth order
belong to $L_2$
in space-time,   detailed statement of Matsumura and 
Nishida theorem is given  in Theorem \ref{thm:NS.1} of Sect. 1 of context. 
This requirement is essentially used  to estimate the $L_\infty$ norm
of necessary order of derivatives in order to enclose the iteration scheme
with the help of Sobolev inequalities and also to treat the material derivatives
of the mass density.

On the other hand, this paper gives the global wellposedness of the same problem
as in \cite{MN2} 
 in $L_2$ in time  and $L_2\cap L_6$ in space maximal regularity 
class, which is an improvement of the Matsumura and Nishida theory in \cite{MN2}
from the point of view of the minimal requirement of  the regularity of solutions.
In fact, after  changing
the material derivatives to time derivatives by  Lagrange 
transformation,  enough estimates obtained by  combination of 
the maximal $L_2$ in time  and $L_2\cap L_6$ in space regularity and 
$L_p$-$L_q$ decay estimate of  the Stokes equations with 
non-slip conditions in the compressible viscous fluid flow case
enable us to use the standard Banach's fixed point argument.

Moreover, one of the purposes of this paper is to present a framework to prove the 
$L_p$-$L_q$ maximal regularity for parabolic-hyperbolic type equations with 
non-homogeneous boundary conditions and how to combine 
the maximal $L_p$-$L_q$ regularity and $L_p$-$L_q$ decay estimates 
of linearized equations to prove the global well-posedness of quasilinear problems
in unbounded domains, which gives a new thought of proving the global 
well-posedness of the initial-boundary value problem for a system of 
parabolic or parabolic-hyperbolic equations with non-homogeneous 
boundary conditions. 
\end{abstract}
{\small Key words and phrases.  Navier-Stokes equations, compressible viscous
barotropic fluid, global well-posedness, \\
 the maximal $L_p$ space}
\section{Introduction}\label{sec:1}

A.~Matsumura and T.~Nishida  \cite{MN2}  proved the existence of  unique
solutions of  equations governing the flow of viscous, compressible, and 
heat conduction fluids in an exterior domain of 3 dimensional Euclidean
space $\BR^3$ for all times, provided the initial data are sufficiently small. 
Although Matsumura and Nishida \cite{MN2} considered the 
 the viscous, barotropic, 
and heat conductive fluid, in this paper we only consider the 
viscous, compressible, barotropic fluid  for simplicity and  reprove
 the Matsumura and Nishida theory in view of the $L_2$ in time and $L_2 \cap L_6$
in space maximal regularity theorem. 

To describe in more detail, we start with 
description of equations  considered in this paper. 
Let $\Omega$ be a three dimensional exterior domain, that is 
the complement, $\Omega^c$, of $\Omega$ is a bounded domain in
the three dimensional Euclidean space $\BR^3$.  Let $\Gamma$ be the 
boundary of $\Omega$, which is a compact $C^2$ hypersurface.  
Let  $\rho=\rho(x, t)$ and  $\bv = (v_1(x, t), 
v_2(x, t), v_3(x, t))^\top$ be respective the mass density and the velocity field, 
where $M^\top$ denotes 
the transposed $M$, $t$ is a time variable and $x=(x_1, x_2, x_3) 
\in \Omega$. Let  $\fp= \fp(\rho)$ be the fluid pressure, which
is a smooth function defined on $(0, \infty)$ such that $\fp'(\rho) > 0$ for $\rho >0$.  
We consider the following equations:
\begin{equation}\label{eq:1.1}\begin{aligned}
\pd_t\rho + \dv(\rho \bv) = 0&&\quad&\text{in $\Omega\times (0, T)$}, \\
\rho(\pd_t\bv + \bv\cdot\nabla\bv) - \DV(\mu\bD(\bv) + \nu \dv\bv\bI
-\fp(\rho)\bI) = 0&&\quad&\text{in $\Omega\times (0, T)$}, \\
\bv|_{\Gamma} = 0, \quad
(\rho, \bv)|_{t=0} = (\rho_* + \theta_0, \bv_0)
&&\quad&\text{in $\Omega$}.
\end{aligned}\end{equation}
Here, $\pd_t = \pd/\pd t$, $\bD(\bv) = \nabla\bv + (\nabla\bv)^\top$ is the deformation tensor, 
$\dv \bv = \sum_{j=1}^3 \pd v_j/\pd x_j$, for a $3\times 3$ matrix
$K$ with $(i, j)$ th component $K_{ij}$, $\DV K =(\sum_{j=1}^3 \pd K_{1j}/\pd x_j,
\sum_{j=1}^3 \pd K_{2j}/\pd x_j, \sum_{j=1}^3 \pd K_{3j}/\pd x_j)^\top$, 
$\mu$ and $\nu$ are two viscous constants such that 
$\mu > 0$ and $\mu + \nu > 0$, and $\rho_*$ is a positive constant
describing the mass density of a 
reference body. 

According to Matsumura and Nishida \cite{MN2}, we have
the  global well-posedness of equations \eqref{eq:1.1} in the $L_2$ framework
stated as follows: 
\begin{thm}[\cite{MN2}]\label{thm:NS.1}
Let $\Omega$ be a three dimensional exterior domain, the boundary
of which is a smooth $2$ dimensional compact hypersurface.  Then, 
there exsits a small number $\epsilon > 0$ such that for any 
initial data $(\theta_0, \bv_0) \in H^3(\Omega)^4$ 
satisfying  smallness condition: $\|(\theta_0, \bv_0)\|_{H^3(\Omega)}
\leq \epsilon$ and compatibility conditions of order 1, that is 
$\bv_0$ and $\pd_t\bv|_{t=0}$ vanish at $\Gamma$, 
problem \eqref{eq:1.1} admits unique solutions $\rho = \rho_*+\theta$
and $\bv$ with 
\begin{align*}
\theta \in C^0((0, \infty), H^3(\Omega)) \cap C^1((0, \infty), H^2(\Omega)),
\quad \nabla \rho \in L_2((0, \infty), H^2(\Omega)^3), \\
\bv \in C^0((0, \infty), H^3(\Omega)^3) \cap C^1((0, \infty), H^1(\Omega)^3),
\quad \nabla \bv \in L_2((0, \infty),, H^3(\Omega)^9). 
\end{align*}
\end{thm}
Matsumura and 
Nishida \cite{MN2} proved Theorem \ref{thm:NS.1} essentially 
 by energy method.   One of key issues in \cite{MN2} is to estimate 
$\sup_{t \in (0, \infty)} \|\bv(\cdot, t)\|_{H^1_\infty(\Omega)}$
by Sobolev's inequality, namely  
\begin{equation}\label{sob:1}
\sup_{t \in ((0, \infty)} \|\bv(\cdot, t)\|_{H^1_\infty(\Omega)}
\leq C\sup_{t \in (0, \infty)} \|\bv(\cdot, t))\|_{H^3(\Omega)}.
\end{equation}

Recently, Enomoto and Shibata \cite{ES2} proved 
the global wellposedness of equations \eqref{eq:1.1} for 
$(\theta_0, \bv_0) \in H^2(\Omega)^4$ with small norms. Namely, 
they proved the following theorem.
\begin{thm}[\cite{ES2}]\label{thm:ES.1} Let $\Omega$ 
be a three dimensional exterior domain, the boundary
of which is a smooth $2$ dimensional compact hypersurface.   Then, 
there exsits a small number $\epsilon > 0$ such that for any 
initial data $(\theta_0, \bv_0) \in H^2(\Omega)^4$ 
satisfying   $\|(\theta_0, \bv_0)\|_{H^2(\Omega)}
\leq \epsilon$ and compatibility condition: $\bv_0|_{\Gamma}=0$, 
problem \eqref{eq:1.1} admits unique solutions $\rho = \rho_*+\theta$
and $\bv$ with 
\begin{align*}
\theta \in C^0((0, \infty), H^2(\Omega)) \cap C^1((0, \infty), H^1(\Omega)),
\quad \nabla \rho \in L_2((0, \infty), H^1(\Omega)^3), \\
\bv \in C^0((0, \infty), H^2(\Omega)^3) \cap C^1((0, \infty), L_2(\Omega)^3),
\quad \nabla \bv \in L_2((0, \infty), H^2(\Omega)^9). 
\end{align*}
\end{thm}
The method used in the proof of Enomoto and Shibata \cite{ES2} is  
essentially the same as that in
Matsumura and Nishida \cite{MN2}.  Only the difference is that 
\eqref{sob:1} is replaced by
$\int^\infty_0\|\nabla\bv\|_{L_\infty(\Omega)}^2\,dt
\leq C\int^\infty_0\|\nabla\bv\|_{H^2(\Omega)}^2\,dt$ 
in \cite{ES2}.
As a conclusion, in the $L_2$ framework the least regularity we need is that   
$\nabla \rho \in L_2((0, \infty), H^1(\Omega)^3)$ and $\nabla \bv \in 
L_2((0, \infty), H^2(\Omega)^9)$. 
In this paper, we improve this point by solving the equations \eqref{eq:1.1} 
in the $L_p$-$L_q$ maximal regularity class, that is 
 the following theorem is a main result of this paper. 
\begin{thm}\label{thm:main0} 
Let $\Omega$ be an exterior domain in $\BR^3$, whose boundary $\Gamma$
is a compact $C^2$ hypersurface and $T \in (0, \infty)$. 
Let $0 < \sigma < 1/6$ and 
$p=2$ or $p=1+\sigma$.  Let $b$ be a number defined by 
$b= (3-\sigma)/2(2+\sigma)$ when $p=2$ and $b= (1-\sigma)/2(2+\sigma)$
when $p=1+\sigma$. Let  $r= 2(2+\sigma)/(4+\sigma)=(1/2+1/(2+\sigma))^{-1}$. 
Set 
\begin{align*}
&\CI = \{(\theta_0, \bv_0) \mid \theta_0 \in ( \bigcap_{q=r, 2, 2+\sigma, 6}
H^1_q(\Omega)),  \quad \bv_0 \in (\bigcap_{q=2, 2+\sigma, 6}
 B^{2(1-1/p)}_{q,p}(\Omega)^3) \cap L_r(\Omega)^3\}, \\
&\|(\theta_0, \bv_0)\|_{\CI} = \sum_{q=2, 2+\sigma, 6} \|\theta_0\|_{H^1_q(\Omega)}
+ \sum_{q=2,2+\sigma, 6}\|\bv_0\|_{B^{2(1-1/p)}_{q,p}(\Omega)} 
+ \|(\theta_0, \bv_0)\|_{H^{1,0}_r(\Omega)}.
\end{align*}
Here and hereafter, we write $\|(\theta, \bv)\|_{H^{\ell, m}_q(\Omega)}
= \|\theta\|_{H^\ell_q(\Omega)} + \|\bv\|_{H^m_q(\Omega)}$ and 
$H^0_q(\Omega) = L_q(\Omega)$. 
 Then, there exists a small constant
$\epsilon \in (0, 1)$ independent of $T$ such that if 
initial data $(\theta_0, \bv_0) \in \CI$ satisfy the compatibility condition:
$\bv_0|_\Gamma=0$ and the smallness condition :
$\|(\theta_0, \bv_0)\|_{\CI} \leq \epsilon^2$, then problem \eqref{eq:1.1}
admits unique solutions $\rho=\rho_*+\theta$ and $\bv$ with
\begin{equation}\label{sol:1}\begin{aligned}
\theta &\in H^1_p((0, T), L_2(\Omega)) \cap L_6(\Omega))
\cap L_p((0, T), H^1_2(\Omega)) \cap H^1_6(\Omega)), \\
\bv & \in H^1_p((0, T), L_2(\Omega)^3 \cap L_6(\Omega)^3) \cap
L_p((0, T), H^2_2(\Omega)^3 \cap H^2_6(\Omega)^3).
\end{aligned}\end{equation}
Moreover, setting 
\begin{align*}
\CE_T(\theta, \bv) &
= \|<t>^b(\theta, \bv)\|_{L_\infty((0, T), L_2(\Omega) \cap L_6(\Omega))}
+ \|<t>^b\nabla(\theta, \bv)\|_{L_p((0, T), H^{0,1}_2(\Omega)
 \cap H^{0,1}_{2+\sigma}(\Omega))} \\
&+ \|<t>^b(\theta, \bv)\|_{L_p((0, T), H^{1,2}_6(\Omega))}
+ \|<t>^b\pd_t(\theta, \bv)\|_{L_p((0, T), L_2(\Omega) \cap L_6(\Omega))},
\end{align*}
we have $\CE_T(\theta, \bv) \leq \epsilon$. 
\end{thm}
\begin{remark} \thetag1~ $T>0$ is taken arbitrarily and $\epsilon>0$ is 
chosen independently of $T$, and so Theorem \ref{thm:main0} 
tells us the global wellposedness of equations \eqref{eq:1.1} for $(0, \infty)$
time inverval.  \\
\thetag2~
In the case $p=2$, Theorem \ref{thm:main0} gives an 
extension of Matsumura and Nishida theorem \cite{MN2}.  
Roughly speaking, if we assume that $(\theta_0, \bv_0) 
\in H^3_2(\Omega)^4$, then $(\theta_0, \bv_0) 
\in (H^1_2(\Omega) \cap H^1_6(\Omega))\times (B^1_{2,2}(\Omega) \cap 
B^1_{6,2}(\Omega))$, and so the global wellposedness holds in
the class as  
$$\theta \in H^1_2((0, T), H^1_2(\Omega) \cap H^1_6(\Omega)), 
\quad \bv \in H^1_2((0, T), L_2(\Omega)^3 \cap L_6(\Omega)^3) 
\cap L_2((0, T), H^2_2(\Omega)^3\cap H^2_6(\Omega)^3)
$$
under the additional condition: 
$(\theta_0, \bv_0) \in H^{1,0}_r(\Omega)$.

On the other hand, choosing $p=1+\sigma$ gives 
the minimal  regularity assumption of initial velocity field  
in the $L_2 \cap L_6$ framework. 
\end{remark}

As  related topics, we consider the Cauchy problem, that is 
$\Omega= \BR^3$ without boundary condition. A. Matsumura and T. Nishida
\cite{MN1} proved the global wellposedness theorem, the statement 
of which is essentially
the same as in Theorem \ref{thm:NS.1} and the proof is based on 
energy method. R.~Danchin \cite{Danchin} proved the global wellposedness
in the critical space by using the Littlewood-Paley decomposition. 
\begin{thm}[\cite{Danchin}]\label{Danchin}
 Let $\Omega= \BR^N$ $(N\geq 2)$.  Assume that $\mu > 0$ and 
$\mu+\nu > 0$.  Let $B^s = \dot B^s_{2,1}(\BR^N)$ and 
$$F^s = (L_2((0, \infty), B^s) \cap C((0, \infty), B^s \cap B^{s-1}))
\times (L_1((0, \infty), B^{s+1}) \cap C((0, \infty), B^{s-1}))^N.
$$
Then, there exists an $\epsilon > 0$ such that if initial data
$\theta_0 \in B^{N/2}(\BR^N) \cap B^{N/2-1}(\BR^N)$ and 
$\bv_0 \in B^{N/2-1}(\BR^N)^N$ satisfy the condition:
$$\|\theta_0\|_{B^{N/2}(\BR^N) \cap B^{N/2-1}(\BR^N)}
+ \|\bv_0\|_{B^{N/2-1}(\BR^N)} \leq \epsilon, $$
then problem \eqref{eq:1.1} with $\Omega=\BR^N$ and $T=\infty$ 
admits a unique solution $\rho=\rho_*+\theta$ and $\bv$ with 
$(\theta, \bv) \in F^{N/2}$.
\end{thm} 
In the case where $\Omega= \BR^3$ or $\BR^N$, there are a lot of works 
concerning \eqref{eq:1.1},  but we do not mention them
any more, 
because we are interested only in the global wellposedness
 in  exterior domains.  
For more information on references,  refer to 
Enomoto and Shibata \cite{ES1}.  

Concerning the $L_1$ in time maximal 
regularity  in exterior domains, the incompressible viscous fluid flows
has been treated by  Danchin and Mucha \cite{DM}.  
To obtain $L_1$ maximal regularity in time, we have to use $\dot B^s_{q,1}$
in space, which is slightly regular space than $H^s_q$, and the 
decay estimates for semigroup on $\dot B^s_{q,1}$ must be needed to 
controle terms arising from the cut-off procedure near the boundary.
Detailed arguments related with thses facts can be found in \cite{DM}.
To treat \eqref{eq:1.1} in an exterior domain in the $L_1$
in time maximal regularity framework, 
we have to prepare not only $L_1$ maximal regularity for
model problems in the whole space and the half space but also decay properties
of semigroup in $\dot B^s_{q,1}$, and so this will be a future work. 
From Theorem \ref{thm:main0},
we may say that problem \eqref{eq:1.1} can be solved in $L_{1+\sigma}$ in time
and $L_2\cap L_6$ in space maximal regularity class for any small 
$\sigma \in (0, 1/6)$. 

The paper is organized as follows. In Sect. 2, equations \eqref{eq:1.1} 
are rewriten in  Lagrange coordinates to eliminate $\bv\cdot\nabla\rho$
and a main result  for  equations with Lagrangian 
description is stated.   
In Sect. 3, we give a $L_p$-$L_q$ maximal regularity theorem 
in some abstract setting. In Sect. 4, we give estimates of nonlinear terms.
In Sect. 5, we prove main results stated in Sect. 2.  
In Sect. 6, 
Theorem \ref{thm:main0} is proved by using a main result in Sect. 2.
In Sect. 7, we discuss the $N$ dimensonal case. 

The main point of our 
proof is to obtain maximal regularity estimates with
decay properties of solutions to linearized
equations, the Stokes equations with non-slip conditions. 
 To explain the idea, we write linearized equations 
as $\pd_t u - Au = f$ and $u|_{t=0}=u_0$ symbolically, 
where $f$ is a function corresponding
to nonlinear terms and $A$ is an closed linear operator with domain $D(A)$. 
We write $u=u_1+u_2$, where $u_1$ is a solution to time shifted
equations:  $\pd_t u_1 + \lambda_1u_1- Au_1 = f$ and $u_1|_{t=0}=u_0$ with some large
positive number $\lambda_1$ and $u_2$ is a solution to compensating equations:
$\pd_t u_2 -Au_2 = \lambda_1u_1$ and $u_2|_{t=0} = 0$. Since the fundamental solutions
to shifted equations have exponential decay properties, $u_1$ has the same 
decay properties as these of nonlinear terms $f$.  Moreover $u_1$ belongs to
the domain of $A$ for all positive time. By Duhamel principle $u_2$ is given by  
$u_2= \lambda_1\int^t_0 T(t-s)u_1(s)\,ds$, where $\{T(t)\}_{t\geq 0}$ is a continuous 
analytic semigroup associated with $A$.  By using $L_p$-$L_q$ decay properties
of $\{T(t)\}_{t\geq 0}$ in the interval $0 < s < t-1$ and standard 
estimates of $C_0$ analytic semigroup: $\|T(t-s)u_0\|_{D(A)} \leq C\|u_0\|_{D(A)}$
for $t-1 < s < t$, where $\|\cdot\|_{D(A)}$ denotes a domain norm, 
we obtain maximal $L_p$-$L_q$ regularity of $u_2$ with decay
properties.  This method seems to be a new thought to prove the 
global wellposedness and to be applicable to many 
quasilinear problems of parabolic type or parabolic-hyperbolic mixture type
appearing in mathematical physics.

To end this section, symbols of functional spaces used in this paper are given. Let 
$L_p(\Omega)$, $H^m_p(\Omega)$ and $B^s_{q,p}(\Omega)$ denote
the standard Lebesgue spaces, Sobolev spaces and Besov spaces,
while their norms are written as $\|\cdot\|_{L_p(\Omega)}$,
$\|\cdot\|_{H^m_p(\Omega)}$ and $\|\cdot\|_{B^s_{q,p}(\Omega)}$. 
We write $H^m(\Omega) = H^m_2(\Omega)$, $H^0_q(\Omega)=L_q(\Omega)$
 and $W^s_q(\Omega)=
B^s_{q,q}(\Omega)$. For any Banach space $X$ with norm $\|\cdot\|_X$, 
$L_p((a, b), X)$ and
$H^m_p((a, b), X)$ denote respective the standard $X$-valued Lebesgue spaces
and Sobolev spaces, while their time weighted norms are defined by 
$$\|<t>^b f\|_{L_p((a, b), X)} 
= \begin{cases} \Bigl(\int^b_a(<t>^b\|f(t)\|_X)^p\,dt\Bigr)^{1/p}
\quad& (1 \leq p < \infty), \\
{\rm esssup}_{t \in (a, b)} <t>^b\|f(t)\|_X\quad &(p=\infty),
\end{cases}
$$
where $<t> = (1 + t^2)^{1/2}$. Let 
$X^n = \{ \bv=(u_1, \ldots, u_n)) \mid u_i \in X \enskip (i=1, \ldots, n)\}$, but we write
$\|\cdot\|_{X^n} = \|\cdot\|_X$ for simplicity. 
Let $H^{\ell, m}_q(\Omega) 
= \{(\rho, \bv) \mid \rho \in H^\ell_q(\Omega), \bv \in H^m_q(\Omega)^3\}$
and $\|(\rho, \bv)\|_{H^{\ell, m}_q(\Omega)} = \|\rho\|_{H^\ell_q(\Omega)}
+ \|\bv\|_{H^m_q(\Omega)}$. 
The letter $C$ denotes generic constants and $C_{a, b, \cdots}$ denotes
that constants depend on quantities $a$, $b$, $\ldots$. $C$ and $C_{a,b, \cdots}$
may change from line to line.

\section{Equations in Lagrange coordinates and statment of main results}
To prove Theorem \ref{thm:main0}, we write equations \eqref{eq:1.1} in  Lagrange
coordinates $\{y\}$. Let $\zeta=\zeta(y, t)$ and $\bu=\bu(y, t)$ be the mass density
and the velocity field in Lagrange coordinates $\{y\}$, and 
for a while we assume that 
\begin{equation}\label{lag:3} \bu \in H^1_p((0, T), L_6(\Omega)^3) \cap 
L_p((0, T), H^2_6(\Omega)^3),.
\end{equation}
and the quantity:
$\|<t>^b\nabla\bu\|_{L_p((0, T), H^1_6(\Omega)}$ is small enough
for some $b > 0$ with $bp' > 1$, where $1/p + 1/p' = 1$. 
We consider
the Lagrange transformation:
\begin{equation}\label{lag:1} x = y + \int^t_0 \bu(y, s)\,ds
\end{equation}
and  assume that  
\begin{equation}\label{lag:2}
\int^T_0\|\nabla\bu(\cdot, t)\|_{L_\infty(\Omega)}\,dt <\delta
\end{equation}
with some small number $\delta > 0$.
If $0 < \delta < 1$, then 
for $x_i = y_i + \int^t_0\bu(y_i, s)\,ds$ we have
$$|x_1-x_2| \geq (1-\int^T_0\|\nabla\bu(\cdot, t)\|_{L_\infty(\Omega)}\,dt)
|y_1-y_2|,$$
and so the correspondence \eqref{lag:1} is one to one.  Moreover, 
applying a method due to 
 Str\"ohmer \cite{Strohmer1}, we see that the correspondence \eqref{lag:1} 
 is a $C^{1+\omega}$ ($\omega \in (0, 1/2)$) diffeomorphism from $\overline{\Omega}$
onto itself for any $t \in (0, T)$.
In fact, let  $J =
\bI + \int^t_0\nabla\bu(y, s)\,ds$, which is the Jacobian of the map
defined by \eqref{lag:1},  and then by Sobolev's imbedding theorem
and H\"older's inequality 
for $\omega \in (0, 1/2)$ we have
\begin{equation}\label{lag:4}
\sup_{t \in (0, T)} \|\int^t_0\nabla\bu(\cdot, s)\,ds\|_{C^{\omega}(\overline{\Omega})}
\leq C_\omega\Bigl(\int^T_0<s>^{-bp'}\,ds\Bigr)^{1/p'}
\Bigl(\int^T_0\|<s>\nabla\bu(\cdot, s)\|_{H^1_6(\Omega)}^p\,ds\Bigr)^{1/p}<\infty
\end{equation}
and we may assume that the right hand side of \eqref{lag:4} is small enough 
and \eqref{lag:2} holds in the process of
constructing a solution. 
By \eqref{lag:1}, we have
$$\frac{\pd x}{\pd y} = \bI + \int^t_0\frac{\pd \bu}{\pd y}(y, s)\,ds, $$
and so choosing $\delta > 0$ small enough, we may assume that 
there exists a $3\times 3$ matrix $\bV_0(\bk)$ of $C^\infty$ functions of 
variables $\bk$ for $|\bk| < \delta$, where $\bk$ is a corresponding
variable to $\int^t_0\nabla\bu\,ds$, 
such that
$\frac{\pd y}{\pd x} = \bI + \bV_0(\bk)$ and $\bV_0(0) = 0$.  Let 
$V_{0ij}(\bk)$ be the $(i, j)$ th component of $3\times 3$ matrix 
$V_0(\bk)$, and then we have 
\begin{equation}\label{lag:5}
\frac{\pd}{\pd x_j} = \frac{\pd}{\pd y_j} + \sum_{j=1}^3V_{0ij}(\bk)\frac{\pd}{\pd y_j}.
\end{equation}
Let $X_t(x) = y$ be the inverse map of Lagrange transform \eqref{lag:1} and 
set $\rho(x, t) = \zeta(X_t(x), t)$ and $\bv(x, t) = \bu(X_t(x), t)$. 
Setting 
$$\CD_\dv(\bk)\nabla\bu = \sum_{i, j=1}^3V_{0ij}(\bk)\frac{\pd u_i}{\pd y_j},$$
we have $\dv\bv = \dv\bu + \CD_\dv(\bk)\bu$. 
Let $\zeta = \rho_* + \eta$, and then 
\begin{align*}
\frac{\pd}{\pd t} \rho + \dv(\rho\bu)= \frac{\pd \eta}{\pd t} + 
(\rho_* + \eta)(\dv\bu + \CD_\dv(\bk)\nabla\bu).
\end{align*}
Setting  
\begin{equation}\label{form:1}
\CD_\bD(\bk)\nabla\bu = \bV_0(\bk)\nabla \bu + (\bV_0(\bk)\nabla \bu)^\top,
\end{equation}
we have 
$\bD(\bv) = \nabla \bv + (\nabla\bv)^\top
= (\bI + \bV_0(\bk))\nabla \bu + ((\bI + \bV_0(\bk))\nabla \bu)^\top
= \bD(\bu) + \CD_\bD(\bk)\nabla\bu$.
Moreover, 
\begin{align*}\DV(\mu\bD(\bv) + \nu\dv\bv\bI)
&= (\bI+\bV_0(\bk))\nabla(\mu (\bD(\bu) + \CD_\bD(\bk)\nabla\bu)
+ \nu (\dv\bu + \CD_\dv(\bk)\nabla\bu) \\
& = \DV(\mu\bD(\bu) + \nu\dv\bu\bI) + \bV_1(\bk)\nabla^2\bu
+ (\bV_2(\bk)\int^t_0\nabla^2\bu\,ds)\nabla\bu
\end{align*}
with
\begin{equation}\label{form:2}\begin{aligned}
\bV_1(\bk)\nabla^2\bu &= \mu \CD_\bD(\bk)\nabla^2\bu 
+ \nu \CD_\dv(\bk)\nabla^2\bu \bI \\
&+ \bV_0(\bk)(\mu\nabla\bD(\bu)
+ \nu\nabla \dv\bu \bI + \mu\CD_\bD(\bk)\nabla^2\bu + 
\nu \CD_\dv(\bk)\nabla^2\bu \bI), \\
(\bV_2(\bk)\int^t_0\nabla\bu\,ds)\nabla\bu
&= (\bI+\bV_0(\bk))(\mu (d_\bk\CD_\bD(\bk)\int^t_0\nabla^2\bu\,ds)\nabla\bu
+\nu (d_\bk\CD_\dv(\bk) \int^t_0\nabla^2\bu\,ds \nabla\bu)\bI.
\end{aligned}\end{equation}
Here, $d_\bk F(\bk)$ denotes the derivative of $F$ with respect to $\bk$. 
Note that $\bV_1(0) = 0$.  Moreover, we write 
\begin{equation}\label{form:3}
\nabla \fp(\rho) 
= \fp'(\rho_*)\nabla\eta +(\fp'(\rho_*+\eta) - \fp'(\rho_*))\nabla\eta
+ \fp'(\rho_*+\eta)\bV_0(\bk)\nabla\theta.
\end{equation}
The material derivative $\pd_t\bv + \bv\cdot\nabla\bv$ is changed to $\pd_t\bu$.

Summing up, we have obtained
\begin{equation}\label{eq:2.1}\begin{aligned}
\pd_t\eta + \rho_*\dv \bu = F(\eta, \bu)&&\quad&\text{in $\Omega \times(0, T)$}, \\
\rho_*\pd_t\bu- \DV(\mu\bD(\bu) 
+ \nu\dv\bu\bI - \fp'(\rho_*)\eta) = \bG(\eta, \bu)&&
\quad&\text{in $\Omega \times(0, T)$}, \\
\bu|_\Gamma=0, \quad (\eta, \bu)|_{t=0} = (\theta_0, \bv_0)&&
\quad&\text{in $\Omega$}. 
\end{aligned}\end{equation}
Here, we have set 
\begin{equation}\label{non:linear1}\begin{aligned}
\bk &= \int^t_0\nabla\bu(\cdot, s)\,ds, \\
F(\eta, \bu)&= \rho_* \CD_\dv(\bk)\nabla\bu 
+ \eta(\dv\bu + \CD_\dv(\bk)\nabla\bu),
\\
\bG(\eta, \bu) &= \eta\pd_t\bu+ \bV_1(\bk)\nabla^2\bu 
+ (\bV_2(\bk)\int^t_0\nabla^2\bu\,ds)\nabla\bu\\
&\qquad - (\fp'(\rho_*+\eta)-\fp'(\rho_*))\nabla\eta
- \fp'(\rho_*+\eta)\bV_0(\bk)\nabla\eta 
\end{aligned}\end{equation}
and $\CD_\dv(\bk)\nabla\bu$, $\bV_1(\bk)$ and $\bV_2(\bk)$ 
have been defined in \eqref{form:1}, \eqref{form:2}
and \eqref{form:3}.  Note that $\CD_\bk(0)=0$, 
$\bV_1(0)=0$ and $g(0, 0)=0$.
The following theorem is a main result in this paper. 
\begin{thm}\label{mainthm:2}
Let $\Omega$ be an exterior domain in $\BR^3$, whose boundary $\Gamma$
is a compact $C^2$ hypersurface.  Let $0 < \sigma < 1/6$ and 
$p=2$ or $p=1+\sigma$.  Let $b$ be a number defined by 
$b= (3-\sigma)/2(2+\sigma)$ when $p=2$ and $b= (1-\sigma)/2(2+\sigma)$
when $p=1+\sigma$. Let  $r= 2(2+\sigma)/(4+\sigma)$ and let  $T \in (0, \infty]$. 
Set 
\begin{align*}
&\CI = \{(\theta_0, \bv_0) \mid \theta_0 \in ( \bigcap_{q=r, 2, 2+\sigma, 6}
H^1_q(\Omega)) \quad \bv_0 \in (\bigcap_{q=2, 2+\sigma, 6}
 B^{2(1-1/p)}_{q,p}(\Omega)^3) \cap L_r(\Omega)^3\}, \\
&\|(\theta_0, \bv_0)\|_{\CI} = \sum_{q=2, 2+\sigma, 6} \|\theta_0\|_{H^1_q(\Omega)}
+ \sum_{q=2,2+\sigma, 6}\|\bv_0\|_{B^{2(1-1/p)}_{q,p}(\Omega)} 
+ \|(\theta_0, \bv_0)\|_{H^{1,0}_r(\Omega)}.
\end{align*}
 Then, there exists a small constant
$\epsilon \in (0, 1)$ independent of $T$ such that if 
initial data $(\theta_0, \bv_0) \in X$ satisfy the compatibility condition:
$\bv_0|_\Gamma=0$ and the smallness condition :
$\|(\theta_0, \bv_0)\|_{\CI} \leq \epsilon^2$, then problem \eqref{eq:2.1}
admits unique solutions $\zeta=\rho_*+\eta$ and $\bu$ with
\begin{equation}\label{sol:1*}\begin{aligned}
\eta &\in H^1_p((0, T), H^1_2(\Omega)) \cap H^1_6(\Omega)), \\
\bu & \in H^1_p((0, T), L_2(\Omega)^3 \cap L_6(\Omega)^3) \cap
L_p((0, T), H^2_2(\Omega)^3 \cap H^2_6(\Omega)^3)
\end{aligned}\end{equation}
possessing the estimate $E_T(\eta, \bu) \leq \epsilon$.  Here, we have set
$$E_T(\eta, \bu) = \CE_T(\eta, \bu) + 
\|<t>^b\pd_t\nabla(\eta, \bu)\|_{L_p((0, T), L_q(\Omega))}$$
and $\CE_T(\eta, \bu)$ is the quantity defined in Theorem \ref{thm:main0}. 
\end{thm}
\begin{remark} 
\thetag1~ The choice of $\epsilon$ is independent of $T>0$, and so
solutions of equations \eqref{eq:2.1} exist for any time $t \in (0, \infty)$. \\
\thetag2~ 
For any natural number $m$, 
$B^m_{q, 2}(\Omega) \subset H^m_q(\Omega)$ for
$2 < q < \infty$ and $B^m_{2,2} = H^m$. 
\\
\thetag3~ The condition: $0 < \sigma < 1/6$  guarantees that 
$bp' > 1$.  \\
\thetag4~ 
Letting $\sigma>0$ be taken a small number such that 
$C^{1+\sigma} \subset H^2_6$, we see that  Theorem \ref{mainthm:2}
implies  
$$\int^T_0\|\bu(\cdot, s)\|_{C^{1+\sigma}(\Omega)}\,ds < \delta $$
with some small number  $\delta > 0$, 
which guarantees that  Lagrange transform given in \eqref{lag:1} is a $C^{1+\sigma}$ 
diffeomorphism on $\Omega$.  Moreover,  Theorem \ref{thm:main0} follows
from Theorem \ref{mainthm:2}, the proof of which will be given in
Sect. 6 below. 
\end{remark}

\section{$\CR$-bounded solution operators}

This section gives a general framework of proving the maximal $L_p$ regularity
($1 < p < \infty$), and so  
problem is formulated in an abstract setting.  Let $X$, $Y$, and $Z$ be
three UMD Banach spaces such that $X \subset Z \subset Y$
and $X$ is dense in $Y$, where the inclusions are continuous. 
 Let $A$ be a closed linear operator from $X$ into $Y$ and 
let $B$ be a linear operator from $X$ into $Y$ and also 
from $Z$ into $Y$.   Moreover, we assume that 
$$\|Ax\|_Y \leq C\|x\|_X, \quad \|Bx\|_Z \leq C\|x\|_X, 
\quad \|Bz\|_Y \leq C\|z\|_Z $$
with some constant $C$ for any $x \in X$ and $z \in Z$.  
  Let $\omega \in (0, \pi/2)$ be a fixed number and set
\begin{align*}
\Sigma_\omega &= \{\lambda \in \BC\setminus\{0\}  
\mid |\arg\lambda| < \pi-\omega\},
\quad 
\Sigma_{\omega, \lambda_0}  = \{\lambda \in \Sigma_\omega 
\mid |\lambda| \geq \lambda_0\}.
\end{align*}
We consider an abstract boundary
value problem with parameter $\lambda \in \Sigma_{\omega, \lambda_0}$:
\begin{equation}\label{1}
\lambda u - A u = f, \quad Bu = g. 
\end{equation}
Here, $Bu = g$ represents boundary conditions,  restrictions like
divergence condition for Stokes equations in the incompressible 
viscous fluid flows case, or both of them. 
The simplest example is the following:
$$\lambda u - \Delta u = f \enskip\text{in $\Omega$}, \quad
\frac{\pd u}{\pd \nu} = g \enskip\text{on $\Gamma$}, \\
$$
where $\Omega$ is a uniform $C^2$ domain in $\BR^N$, $\Gamma$ its boundary, 
$\nu$ the unit outer normal to $\Gamma$, and $\pd/\pd\nu = \nu\cdot\nabla$
with $\nabla = (\pd/\pd x_1, \ldots, \pd/\pd x_N)$ for $x=(x_1, \ldots, x_N) \in \BR^N$. 
In this case, it is standard to choose $X = H^2_q(\Omega)$, $Y = L_q(\Omega)$, 
$Z = H^1_q(\Omega)$ with $1 < q < \infty$, $A = \Delta$, and $B = \pd/\pd\nu$. 

Problem formulated in \eqref{1} is corresponding to parameter elliptic problems 
 which have been studied by Agmon \cite{Agmon}, 
Agmon, Douglis and Nirenberg \cite{ADN},
Agranovich and Visik \cite{AV}, Denk and Volevich \cite{DV} and 
references there in, 
and their arrival point is to prove the unique existence of solutions possessing
the estimate:
$$
|\lambda|\|u\|_Y + \|u\|_X 
\leq C(\|f\|_Y + |\lambda|^\alpha\|g\|_Y + \|g\|_Z)
$$
for some $\alpha \in \BR$.  From this estimate, 
we can derive the generation of a $C^0$ analytic 
semigroup associated with $A$ when $Bu=0$.  
But to prove the maximal $L_p$ regularity with $1 < p < \infty$
for the corresponding nonstationary problem:
\begin{align}\label{2}
&\pd_t v- A v = f, \quad Bv= g \quad\text{for $t > 0$}, \quad v|_{t=0} = v_0,
\end{align}
especially in the cases where $Bv=g \not=0$, 
further consideration is needed.  Below, we introduce a framework based on 
 the Weis operator valued Fourier multiplier theorem.  
To state this theorem, we make a preparation. 
\begin{dfn} Let $E$ and $F$ be two Banach spaces and let $\CL(E, F)$ be the 
set of all bounded linear operators from $E$ into $F$. We say that 
an operator family $\CT \subset \CL(E, F)$ is $\CR$ bounded if 
there exist a constant $C$ and an exponent $q \in [1, \infty)$ such that for any
integer $n$, $\{T_j\}_{j=1}^n \subset \CT$ and $\{f_j\}_{j=1}^n
\subset E$, the inequality:
$$\int^1_0\|\sum_{j=1}^n r_j(u)T_jf_j\|_F^q\,d u
\leq C\int^1_0\|\sum_{j=1}^n r_j(u)f_j\|_E^q\,du$$
is valid, where the Rademacher functions $r_k$, $k \in \BN$, are
given by $r_k: [0, 1] \to \{-1, 1\}$; $t \mapsto {\rm sign}(\sin 2^k\pi t)$.
The smallest such $C$ is called $\CR$ bound of $\CT$ on $\CL(E, F)$, 
which is denoted by $\CR_{\CL(E, F)}\CT$.
\end{dfn}
For $m(\xi) \in L_\infty(\BR\setminus\{0\}, \CL(E, F))$, we set
$$T_mf = \CF^{-1}_\xi[m(\xi)\CF[f](\xi)] \quad f \in \CS(\BR, E),$$
where $\CF$ and $\CF_\xi^{-1}$ denote respective Fourier transformation and 
inverse Fourier transformation.
\begin{thm}[Weis's operator valued Fourier multiplier theorem]
Let $E$ and $F$ be two UMD Banach spaces. 
Let $m(\xi) \in C^1(\BR\setminus\{0\}, \CL(E, F))$ and assume that 
\begin{align*}
\CR_{\CL(E, F)}(\{m(\xi) &\mid \xi \in \BR\setminus\{0\}\}) \leq r_b \\
\CR_{\CL(E, F)}(\{\xi m'(\xi) &\mid \xi \in \BR\setminus\{0\}\}) \leq r_b
\end{align*}
with some constant $r_b > 0$.  Then, for any $p \in (1, \infty)$, 
$T_m \in \CL(L_p(\BR, E), L_p(\BR, F))$ and 
$$\|T_mf\|_{L_p(\BR, F)} \leq C_pr_b\|f\|_{L_p(\BR, E)}
$$
with some constant $C_p$ depending solely on $p$.
\end{thm}
\begin{remark} For a proof, refer to Weis \cite{Weis}. 
\end{remark}
We introduce the following assumption.   Recall that $\omega$ is a fixed number
such that $0 < \omega < \pi/2$. 
\begin{assump}\label{assump:1}  Let  $X$, $Y$ and $Z$ 
be UMD Banach spaces. There exist a constant $\lambda_0$, 
$\alpha \in \BR$, and an operator family $\CS(\lambda)$ with 
$$\CS(\lambda)\in {\rm Hol}\,(\Sigma_{\omega,\lambda_0}, 
\CL(Y\times Y \times Z, X))$$
such that for any $f \in Y$ and $g \in Z$, $u=\CS(\lambda)(f, \lambda^\alpha g, g)$
is a solution of equations \eqref{1}, and the estimates:
\begin{align*}
\CR_{\CL(Y\times Y \times Z, X)}(\{(\tau\pd_\tau)^\ell\CS(\lambda) \mid
\lambda \in \Sigma_{\omega,\lambda_0}\}) &\leq r_b \\
\CR_{\CL(Y\times Y \times Z, Y)}(\{(\tau\pd_\tau)^\ell(\lambda\CS(\lambda)) \mid
\lambda \in \Sigma_{\omega,\lambda_0}\}) &\leq r_b
\end{align*}
for $\ell=0,1$ are valid, 
where $\lambda = \gamma + i\tau \in \Sigma_{\omega, \lambda_0}$.
$\CS(\lambda)$ is called an $\CR$-bounded solution operator 
or an $\CR$ solver of equations \eqref{1}. 
\end{assump}

We now consider an initial-boundary value problem:
\begin{equation}\label{4}
\pd_t u - Au = f \quad Bu = g \quad(t>0), \quad u|_{t=0} = u_0.
\end{equation}
This problem is divided into the following two equations:
\begin{alignat}3
\pd_t u - Au &= f &\quad Bu &= g &\quad&(t \in \BR); \label{eq:1} \\
\pd_t u - Au &= 0 &\quad Bu &= 0 &\quad&(t>0), \quad u|_{t=0} = u_0.
\label{eq:2}
\end{alignat}
From the definition of $\CR$-boundedness with $n=1$ we see that 
$u=\CS(\lambda)(\bff, 0, 0)$ satisifes  equations:
$$\lambda u -Au = f, \quad Bu = 0,$$
and the estimate:
$$|\lambda|\|u\|_Y + \|u\|_X \leq C\|f\|_Y.$$
Thus, $A$ generates a $C^0$ analytic semigroup $\{T(t)\}_{t\geq 0}$ such that
$u = T(t)u_0$ solves equations \eqref{eq:2} uniquely and 
\begin{equation}\label{semi.est.1}\|u(t)\|_Y \leq r_be^{\lambda_0t}\|u_0\|_Y, 
\quad \|\pd_tu(t)\|_Y \leq r_be^{\lambda_0 t}\|u_0\|_Y,
\quad \|\pd_tu(t)\|_Y \leq r_be^{\lambda_0 t}\|u_0\|_X.
\end{equation}
These estimates and trace method of real-interpolation theory  
yield the following theorem.
\begin{thm}[Maximal regularity for initial value problem] \label{max.thm.1}
Let $1 < p < \infty$ and set $\CD= (Y, X_B)_{1-1/p, p}$, 
where $X_B = \{u_0 \in X \mid Bu_0=0\}$, and  $(\cdot, \cdot)_{1-1/p, p}$
denotes a real interpolation functor. Then, for any $u_0 \in \CD$, 
problem \eqref{eq:2} admits a unique solution $u$ with
$$e^{-\lambda_0t}u \in L_p(\BR_+, X) \cap H^1_p(\BR_+, Y)
\quad(\BR_+=(0, \infty))$$
possessing the estimate:
$$\|e^{-\lambda_0t}\pd_tu\|_{L_p(\BR_+, Y)}
+ \|e^{-\lambda_0t}u\|_{L_p(\BR_+, X)} \leq C\|u_0\|_{(Y, X)_{1-1/p, p}}.
$$ 
\end{thm}
The $\CR$-bounded solution operator plays an essential role  to prove 
the following theorem. 
\begin{thm}[Maximal regularity for boundary value problem] \label{max.thm.2}
Let $1 < p < \infty$.  Then for any $f$ and $g$ with 
$e^{-\gamma t}f \in L_p(\BR, Y)$ and $
e^{-\gamma t}g \in L_p(\BR, Z) \cap H^\alpha_p(\BR, Y)$ 
for any $\gamma > \lambda_0$,  problem \eqref{eq:1}
admits a unique solution $u$ with 
$e^{-\gamma t} u \in L_p(\BR, X) \cap H^1_p(\BR, Y)$
for any $\gamma > \lambda_0$ possessing the estimate:
\begin{align*}
&\|e^{-\lambda_0t}\pd_tu\|_{L_p(\BR_+, Y)}
+ \|e^{-\lambda_0t}u\|_{L_p(\BR_+, X)} \leq C(
\|e^{-\gamma t}f\|_{L_p(\BR, Y)} \\
&\quad + (1+\gamma)^\alpha\|e^{-\gamma t}g\|_{H^\alpha_p(\BR, Y)}
+ \|e^{-\gamma t}g\|_{L_p(\BR, Z)})
\end{align*}
for any $\gamma > \lambda_0$. Here, the constant 
$C$ may depend on $\lambda_0$ but independent of 
$\gamma$ whenever $\gamma > \lambda_0$, and we have set
$$H^\alpha_p(\BR, Y) = \{h \in \CS'(\BR, Y) \mid 
\|h\|_{H^\alpha_p(\BR, Y)} : = 
\|\CF^{-1}_\xi[(1+|\xi|^2)^{\alpha/2}\CF[f](\xi)]\|_{L_p(\BR, Y)} < \infty\}.
$$
\end{thm}
\begin{proof}
Let $\CL$ and $\CL^{-1}$ denote respective Laplace transformation and 
inverse Laplace transformation defined by setting
\begin{align*}\CL[f](\lambda) &= \int_\BR e^{-\lambda t}f(t)\,d t = 
\int_\BR e^{-i\tau t}(e^{-\gamma t}f(t))\,d t
= \CF[e^{-\gamma t}f(t)](\tau) \quad(\lambda = \gamma + i\tau),\\
\CL^{-1}[f](t) &= \frac{1}{2\pi}\int_\BR e^{\lambda t}f(\tau)\,d \tau = 
\frac{e^{\gamma t}}{2\pi}\int_\BR e^{-i\tau t}f(\tau)\,d \tau
= e^{\gamma t}\CF^{-1}[f](\tau). 
\end{align*}
We consider equations:
$$\pd_tu-Au = f, \quad Bu = g \quad\text{for $t \in \BR$}.$$
Applying Laplace transformation yields that
$$\lambda \CL[u](\lambda) -A\CL[u](\lambda) =\CL[ f](\lambda), 
\quad B\CL[u](\lambda) = \CL[g](\lambda). 
$$
Applying $\CR$-bounded solution operator $\CS(\lambda)$ yields that
$$\CL[u](\lambda) = \CS(\lambda)(\CL[f](\lambda), 
\lambda^\alpha\CL[g](\lambda), \CL[g](\lambda)),
$$
and so
$$u = \CL^{-1}[\CS(\lambda)\CL[(f, \Lambda^\alpha g, g)](\lambda)],$$
where $\Lambda^\alpha g = \CL^{-1}[\lambda^\alpha \CL[g]]$.
Moreover, 
$$\pd_tu =  \CL^{-1}[\lambda\CS(\lambda)\CL[f, \Lambda^\alpha g, g)](\lambda)].
$$
Using Fourier transformation and inverse Fourier transformation, we rewrite
\begin{align*}
u &= e^{\gamma t}\CF^{-1}[\CS(\lambda)\CF[e^{-\gamma t}(f, \Lambda^\alpha g, g)]
(\tau)](t), \\
\pd_tu &= e^{\gamma t}\CF^{-1}[\lambda\CS(\lambda)
\CF[e^{-\gamma t}(f, \Lambda^\alpha g, g)](\tau)](t).
\end{align*}
Applying the assumption of $\CR$-bounded solution operators 
and Weis's operator valued Fourier
multiplier theorem yields that 
\begin{align*}
&\|e^{-\gamma t}\pd_tu\|_{L_p(\BR, Y)} + \|e^{-\gamma t}u\|_{L_p(\BR, X)}\\
&\quad
\leq C_pr_b(\|e^{-\gamma t}f\|_{L_p(\BR, Y)} 
+ (1 + \gamma)^\alpha\|e^{-\gamma t}g\|_{H^\alpha_p(\BR, Y)}
+ \|e^{-\gamma t}g\|_{L_p(\BR, Z)})
\end{align*}
for any $\gamma > \lambda_0$.  The uniqueness follows from the generation of
analytic semigroup and Duhamel's principle. 
\end{proof}

We now consider a time shifted
equations:
\begin{equation}\label{eq:3}
\pd_tu + \lambda_1 u - Au = f, \quad Bu = g \quad\text{for $t \in (0, \infty)$},
\quad u|_{t=0} = u_0. 
\end{equation}
As a first step, we consider the following  time shifted equations without initial data
\begin{equation}\label{eq:4}
\pd_tu + \lambda_1 u - Au = f, \quad Bu = g \quad\text{for $t \in \BR$}. 
\end{equation}
Then, we have the following theorem which guarantees the polynomial decay
of solutions.  
\begin{thm}\label{max:thm.2*}
Let $\lambda_0$ be a constant appearing in Assumption \ref{assump:1} and 
let $\lambda_1 > \lambda_0$. 
Let $1 < p < \infty$ and $b \geq 0$.  
Then, for any $f$ and $g$ with 
$<t>^bf \in L_p(\BR, Y)$ and $<t>^bg \in L_p(\BR, Z) \cap H^\alpha_p(\BR, X)$, 
problem \eqref{eq:4} admits a unique solution
$w \in H^1_p((0, \infty), Y) \cap L_p((0, \infty), X)$ possessing the estimate:
\begin{equation}\label{eq:8}\begin{aligned}
&\|<t>^bw\|_{L_p((0, \infty), X)} + \|<t>^b\pd_tw\|_{L_p((0, \infty), Y)} \\
&\quad \leq C(\|<t>^bf\|_{L_p(\BR, Y)}
+ \|<t>^bg\|_{H^\alpha_p(\BR, Y)} + \|<t>^bg\|_{L_p(\BR, Z)}).
\end{aligned}\end{equation}
\end{thm}
\begin{proof}
Since $ik + \lambda_1 \in \Sigma_{\omega, \lambda_0}$,  for $k \in \BR$ 
we set 
$w = \CF^{-1}[\CM(ik + \lambda_1)(\CF[f], (ik)^\alpha \CF[g], \CF[g])]$,
and then $w$ satisfies equations: 
$$\pd_tw + \lambda_1w - Aw=f, \quad Bw=g \quad\text{for $t \in \BR$},
$$
and the estimate:
\begin{equation}\label{est:3.1}
\|\pd_t w\|_{L_p(\BR, Y)} + \|w\|_{L_p(\BR, X)} \leq C(\|f\|_{L_p(\BR, Y)}
+ \|g\|_{H^\alpha_p(\BR, Y)} + \|g\|_{L_p(\BR, Z)}). 
\end{equation}
This prove the theorem in the case where $b=0$.  When $0 < b \leq 1$, 
we observe that 
$$\pd_t(<t>^bw) +\lambda_1(<t>^bw) - A(<t>^bw) = <t>^bf
+ <t>^{b-2}tw, \quad B(<t>^bw) = <t>^bg,$$
and so noting that $\|<t>^{b-2}t w\|_Y \leq C\|w\|_Y \leq C\|w\|_X$, 
we have
\begin{align*}
&\|<t>^bw\|_{L_p((0, \infty), X)} + \|<t>^b\pd_tw\|_{L_p((0, \infty), Y)} \\
&\quad \leq C(\|<t>^{b-2}tw\|_{L_p(\BR, Y)} + \|<t>^bf\|_{L_p(\BR, Y)}
+ \|<t>^bg\|_{H^\alpha_p(\BR, Y)} + \|<t>^bg\|_{L_p(\BR, Z)}) \\
&\quad  \leq C( \|<t>^bf\|_{L_p(\BR, Y)}
+ \|<t>^bg\|_{H^\alpha_p(\BR, Y)} + \|<t>^bg\|_{L_p(\BR, Z)}). 
\end{align*}
If $b > 1$, then repeated use of this argument yields the theorem, which 
completes the proof of Theorem \ref{max:thm.2*}.
\end{proof}

Finally, we consider equations \eqref{eq:3}.  Let $w$ be a  solution of
\eqref{eq:4}, the unique existence of which is guaranteed by Theorem 
\ref{max:thm.2*}.  Let $v= u-w$, and then $v$ satisifes equations:
\begin{equation}\label{eq:5}
\pd_tv + \lambda_1 v - Av = 0, \quad Bv = 0 \quad\text{for $t \in (0, \infty)$},
\quad v|_{t=0} = u_0-w|_{t=0}
\end{equation}
Let $\{T(t)\}_{t\geq0}$ be a continuous analytic semigroup 
satisfying \eqref{semi.est.1}. Set $u_1 = u_0-w|_{t=0}$ and  
$v = e^{-\lambda_1 t}T(t)u_1$,  and then 
\begin{gather} \pd_t v + \lambda_1v - Av=0, \quad Bv=0, 
\quad v|_{t=0} = u_1, \label{eq:6} \\
\|v(t)\|_Y \leq r_be^{-(\lambda_1-\lambda_0)t}
\|u_1\|_Y, \, \|\pd_t v(t)\|_Y \leq r_be^{-(\lambda_1-\lambda_0)t}\|u_1\|_Y,
\, \|\pd_t v(t)\|_Y \leq r_be^{-(\lambda_1-\lambda_0)t}\|u_1\|_X.
\label{eq:7}
\end{gather}
Thus, the  trace method of real interpolation theory yields the following theorem. 
\begin{thm}\label{max.thm.1*}
Let $1 < p < \infty$ and $b>0$. Let $\CD$ be the same space as in Theorem 
\ref{max.thm.1}.  If $u_1 \in \CD$ and $f$ and $g$ satisfy the same condition
as in Theorem \ref{max:thm.2*},  then 
problem \eqref{eq:3} admits a unique solution 
$u \in L_p(\BR_+, X) \cap H^1_p(\BR_+, Y)$ $(\BR_+=(0, \infty))$
possessing the estimate:
\begin{equation}\label{eq:9}\begin{aligned}
&\|<t>^b\pd_tu\|_{L_p(\BR_+, Y)}
+ \|<t>^bu\|_{L_p(\BR_+, X)} \\
&\quad  \leq C(\|u_0\|_{(Y, X)_{1-1/p, p}}
+ \|<t>^bf\|_{L_p(\BR, Y)}
+ \|<t>^bg\|_{H^\alpha_p(\BR, Y)} + \|<t>^bg\|_{L_p(\BR, Z)}). 
\end{aligned}\end{equation}
\end{thm}
\begin{proof}
Let $v = e^{-\lambda_1 t}\,T(t)u_1$, and then $v$ satisfies equations
\eqref{eq:6}.  Since $u_1 \in \CD$, by trace method of real interpolation
theorem and \eqref{eq:7}, we have 
\begin{equation}\label{est:3.2}
\|e^{(\lambda_1-\lambda_0)}v\|_{L_p((0, \infty), X)} 
+ \|e^{(\lambda_1-\lambda_0)}\pd_tv\|_{L_p((0, \infty), Y)} 
\leq C\|u_1\|_{(Y, X)_{1-1/p,p}}.
\end{equation}
Since $w$ satisfies \eqref{eq:8}, 
trace method of real interpolation theory yields that  
\begin{equation}\label{est:3.3}\begin{aligned}
\|w\|_{(Y, X)_{1-1/p,p}} &\leq C(\|w\|_{L_p((0, \infty), X)}
+ \|\pd_tw\|_{L_p((0, \infty), Y)}) \\
& \leq C(\|<t>^bf\|_{L_p(\BR, Y)}
+ \|<t>^bg\|_{H^\alpha_p(\BR, Y)} + \|<t>^bg\|_{L_p(\BR, Z)}), 
\end{aligned}\end{equation}
because $b \geq 1$. 
Thus, $u = v+w$ satisfies equations \eqref{eq:3} 
and the estimate \eqref{eq:9}.   The uniqueness of solutions follows
from the generation of continous analytic semigroup and Duhamel's principle.
This completes the proof 
of Theorem \ref{max.thm.1*}. \end{proof}
\section{Estimates of nonlinear terms}

In what follows, let $T > 0$ be any positive time and let $b$ and $p$ be
positive numbers and an exponents given in Theorem \ref{thm:main0}
and Theorem \ref{mainthm:2}. 
Let $\CU^i_\epsilon$ ($i=1,2$) be  underlying spaces for linearized equations of 
equations \eqref{eq:2.1}, which is defined by
\begin{equation}\label{eq:5.1}\begin{aligned}
&\CU^1_T = \{\theta
\in H^1_p((0, T), H^1_2(\Omega) \cap H^1_6(\Omega))\mid \theta|_{t=0}=\theta_0,
\quad \sup_{t \in (0, T)}\|\theta(\cdot,t)\|_{L_\infty(\Omega)} \leq \rho_*/2\},\\
&\CU^2_T = \{\bv \in L_p((0, T), H^2_2(\Omega)^3 
\cap H^2_6(\Omega)^3) \cap 
H^1_p((0, T), L_2(\Omega)^3 \cap L_6(\Omega)^3) \mid  \\
&\hskip6.7cm \bv|_{t=0} = \bv_0, \quad 
\int^T_0\|\nabla\bv(\cdot, s)\|_{L_\infty(\Omega)}\,ds \leq \delta\}.
\end{aligned}\end{equation}
Recall that  our energy $E_T(\eta,\bu)$ has been defined by
\begin{align*}
E_T(\eta, \bu)& =
\|<t>^{b}\nabla(\eta, \bu)\|_{L_p((0, T), H^{0,1}_2(\Omega) \cap H^{0,1}_{2+\sigma}
(\Omega))} 
+  \|<t>^{b}(\eta, \bu)\|_{L_\infty((0, T), L_2(\Omega)\cap L_6(\Omega))} 
\\
&+ \|<t>^{b}\pd_t(\eta, \bu)\|_{L_p((0, T), H^{1,0}_2(\Omega) \cap H^{1,0}_{6}
(\Omega))} 
+\|<t>^{b}(\eta, \bu)\|_{L_p((0, T), H^2_6(\Omega))}. 
\end{align*}
Note that by using a standard interpolation inequality we have
\begin{equation}\label{eq:5.2}
\|f\|_{L_{2+\sigma}(\Omega)} \leq \|f\|_{L_2(\Omega)}^{1-\sigma/4}
\|f\|_{L_6(\Omega)}^{\sigma/4}.
\end{equation}
And therefore, for $(\theta, \bv) \in \CU^1_T\times\CU^2_T$, 
we know that 
\begin{equation}\label{eq:5.3}\begin{aligned}
\|<t>^b(\theta, \bv)\|_{L_\infty((0, T), L_{2+\sigma}(\Omega)}
&\leq C_\sigma \sum_{q=2, 6}\|<t>^b(\theta, \bv)
\|_{L_\infty((0, T), L_q(\Omega))},
\\
\|<t>^b\pd_t(\theta, \bv)\|_{L_p((0, T), H^{1,0}_{2+\sigma}(\Omega)}
&\leq C_\sigma \sum_{q=2, 6}\|<t>^b\pd_t(\theta, \bv)\|_{L_p((0, T), H^{1,0}_q(\Omega))},
\end{aligned}\end{equation}
where $bp' > 1$. 
 Notice that for any $\theta \in \CU^1_T$
we have 
\begin{equation}\label{cond:1}
\rho_*/2 \leq |\rho_*+\tau\theta(y, t)| \leq 3\rho_*/2 \quad\text{for
$(y,t) \in \Omega\times(0, T)$ and $|\tau| \leq 1$}.
\end{equation}
For $\bv \in \CU^2_T$ let  
$\bk_\bv= \int^t_0\nabla\bv(\cdot,s)\,ds$, and then 
$|\bk_\bv(y, t)| \leq \delta$ for any $(y, t) \in \Omega\times(0, T)$. 
Moreover, for $q=2, 2+\sigma$ and $6$ by H\"older's inequality 
\begin{equation}\label{eq:4.4}
\sup_{t \in (0, T)}
\|\bk_\bv\|_{H^1_q(\Omega)} \leq \int^T_0\|\nabla\bv(\cdot, t)\|_{H^1_q(\Omega)} 
\leq C\Bigl(\int^\infty_0<t>^{-p'b}\Bigr)^{1/p'}\|<t>^b\nabla\bv\|_{L_p((0, T),
H^1_q(\Omega))},
\end{equation}
where $bp' > 1$. 

In what follows, for notational simplicity we use the following 
abbreviation: $\|f\|_{H^1_q(\Omega)} = \|f\|_{H^1_q}$,
$\|f\|_{L_q(\Omega)} = \|f\|_{L_q}$, $\|f\|_{L_\infty((0, T), X)}
= \|f\|_{L_\infty(X)}$, and $\|<t>^bf\|_{L_p((0, T), X)}
= \|f\|_{L_{p,b}(X)}$. 
Let $(\theta, \bv) \in \CU^1_T\times\CU^2_T$ and 
$(\theta_i, \bv_i) \in \CU^1_T\times \CU^2_T$ ($i=1,2$). 
The purpose of this section is to give necessary estimates of
$(F(\theta, \bv), \bG(\theta, \bv))$ and difference:
$(F(\theta_1, \bv_1) -F(\theta_2, \bv_2), \bG(\theta_1, \bv_1)
- \bG(\theta_2, \bv_2)))$ to prove the global wellposedness of 
equations \eqref{eq:2.1}. 
Recall that
\begin{equation}\label{nonlinear:4.1}\begin{aligned}
F(\theta, \bv) &= \rho_*\CD_\dv(\bk)\nabla\bv +\theta\dv\bv + \theta\CD_\dv(\bk)\nabla\bv,
\\
\bG(\theta, \bv) &= \theta\pd_t\bv+ \bV_1(\bk)\nabla^2\bv 
+ (\bV_2(\bk)\int^t_0\nabla^2\bv\,ds)\nabla\bv\\
&\qquad - (\fp'(\rho_*+\theta)-\fp'(\rho_*))\nabla\theta
- \fp'(\rho_*+\theta)\bV_0(\bk)\nabla\theta.
\end{aligned}\end{equation}
We start with estimating $\|F(\theta, \bv)\|_{L_{p,b}(H^1_r)}$. 
Recall that  $r^{-1} = 2^{-1} + (2+\sigma)^{-1}$ and we use the  estimates:
\begin{equation}\label{geneq:4.1}\begin{aligned}
\|fg\|_{L_{p,b}(H^1_r)} &\leq C\|f\|_{L_\infty(H^1_{2+\sigma})}
\|g\|_{L_{p,b}(H^1_2)}, \\
\|fgh\|_{L_{p,b}(H^1_r)} &\leq C(\|f\|_{L_\infty(H^1_6)}
\|g\|_{L_\infty(H^1_{2+\sigma})}
+ \|f\|_{L_\infty(H^1_{2+\sigma})}\|g\|_{L_\infty(H^1_6)})
\|h\|_{L_{p,b}(H^1_2)}, 
\end{aligned}\end{equation}
as follows from H\"older's inequality and Sobolev's inequality :
$\|f\|_{L_\infty} \leq C\|f\|_{H^1_6}$. 
Let $dG(\bk)$ denote the derivative of $G(\bk)$ with
respect to $\bk$ and  $C_\dv$ be a constan such that 
$\sup_{|\bk| < \delta}|\CD_\dv(\bk)| < C_\dv$, 
$\sup_{|\bk| < \delta}|d\CD_\dv(\bk)| < C_\dv$, and 
$\sup_{|\bk| < \delta}|d(d\CD_\dv)(\bk)| < C_\dv$.
Then,  noting $\CD_\dv(0)=0$,  by \eqref{eq:4.4}  we have
\begin{equation}\label{eq:4.0}\begin{aligned}
\|\CD_\dv(\bk_{\bv})\|_{H^1_q}
&\leq C_\dv \|\bk_{\bv}\|_{H^1_q} \leq C\|\nabla\bv\|_{L_{p,b}(H^1_q)}
\quad\text{for $\bv
\in \CU^2_T$ and $q=2, 2+\sigma$ and $6$}.
\end{aligned}\end{equation}
Moreover, for $\bv_1$, $\bv_2 \in \CU^2_T$ writing 
$$\CD_\dv(\bk_{\bv_1}) - \CD_\dv(\bk_{\bv_2})
= \int^t_0d\CD_\dv(\bk_{\bv_2}+ \tau(\bk_{\bv_1}-\bk_{\bv_2}))\,d\tau
\,(\bk_{\bv_1}-\bk_{\bv_2}), $$
and noting that 
$|\bk_{\bv_2}+ \tau(\bk_{\bv_1}-\bk_{\bv_2})| = |(1-\tau)\bk_{\bv_2}
+ \tau \bk_{\bv_1}| \leq (1-\tau)\delta + \tau \delta = \delta$, 
we have
\begin{equation}\label{eq:4.1}\begin{aligned}
&\|\CD_\dv(\bk_{\bv_1})- \CD_\dv(\bk_{\bv_2})\|_{H^1_q}\\
&\quad\leq C_\dv (\|\bk_{\bv_1}-\bk_{\bv_2}\|_{L_\infty(H^1_q)}
+ \sum_{i=1,2}\|\nabla \bk_{\bv_i}\|_{L_\infty(L_q)}
\|\bk_{\bv_1}-\bk_{\bv_2}\|_{L_\infty(L_\infty)})\\
&\quad \leq C(\|\nabla(\bv_1-\bv_2)\|_{L_{p,b}(H^1_q)}
+ \sum_{i=1,2}\|\nabla \bv_i\|_{L_{p,b}(H^1_q)}
\|\nabla(\bv_1-\bv_2)\|_{L_{p,b}(H^1_6)}. 
\end{aligned}\end{equation}
Since $\theta= \theta|_{t=0} + \int^t_0\pd_s\theta\,ds$,
for $X \in \{L_q, H^1_q\}$ with $q=2$, $2+\sigma$ and $6$  
\begin{equation}\label{eq:4.2}\begin{aligned}
\|\theta(\cdot, t)\|_{X} &\leq \|\theta_0\|_{X} 
 + \int^T_0\|(\pd_s\theta)(\cdot, s)\|_{X}\,ds
\\
&\leq\|\theta_0\|_{X} + \Bigl(\int^\infty_0<t>^{-p'b}\,dt\Bigr)^{1/p'}
\|\pd_s\theta\|_{L_{p,b}(X)}.
\end{aligned}\end{equation}
In particular, by Sobolev's inequality 
\begin{equation}\label{eq:4.3}
\|\theta(\cdot, t)\|_{L_\infty}  \leq C(\|\theta_0\|_{H^1_6}
+ \|\pd_t\theta\|_{L_{p,b}(H^1_{6})}).
\end{equation}
For $\theta \in \CU^1_T$ and $\bv \in \CU^2_T$, 
combining \eqref{geneq:4.1}, \eqref{eq:4.0}, 
\eqref{eq:4.1},
\eqref{eq:4.2}, and \eqref{eq:4.3} yields that
\begin{equation}\label{nones:4.1}\begin{aligned}
&\|F(\theta, \bv)\|_{L_{p,b}(H^1_r)} \leq C[\|\nabla\bv\|_{L_{p,b}(H^1_{2+\sigma})}
\|\nabla\bv\|_{L_{p,b}(H^1_2)} 
+ (\|\theta_0\|_{H^1_{2+\sigma}}+\|\pd_t\theta\|_{L_{p,b}(H^1_{2+\sigma})})
\|\nabla\bv\|_{L_{p,b}(H^1_2)}\\
&\quad +\{(\|\theta_0\|_{H^1_6}+\|\pd_t\theta\|_{L_{p,b}(H^1_6)})
\|\nabla\bv\|_{L_{p,b}(H^1_{2+\sigma})} 
+ (\|\theta_0\|_{H^1_{2+\sigma}}+\|\pd_t\theta\|_{L_{p,b}(H^1_{2+\sigma})})	
\|\nabla\bv\|_{L_{p,b}(H^1_6)}\}\\
&\hskip12.4cm\times\|\nabla\bv\|_{L_{p,b}(H^1_2)}].
\end{aligned}\end{equation}
Analogously, for  $\theta_i \in \CU^1_T$ and $\bv_i \in \CU^2_T$
($i=1,2$),
\allowdisplaybreaks
\begin{align}
&\|F(\theta_1, \bv_1)-F(\theta_2, \bv_2)\|_{L_{p,b}(L_r)} \nonumber \\
&\leq C[(\|\nabla(\bv_1-\bv_2)\|_{L_{p,b}(H^1_{2+\sigma})} 
+ \sum_{i=1,2}\|\nabla\bv_i\|_{L_{p,b}(H^1_{2+\sigma})}
\|\nabla(\bv_1-\bv_2)\|_{L_{p,b}(H^1_6)})\|\nabla\bv_1\|_{L_{p,b}(H^1_2)}
\nonumber\\
&+ \|\nabla\bv_2\|_{L_{p,b}(H^1_{2+\sigma})}\|\nabla(\bv_1-\bv_2)\|_{L_{p,b}(H^1_2)}
+ \|\pd_t(\theta_1-\theta_2)\|_{L_{p,b}(H^1_{2+\sigma})}
\|\nabla\bv_1\|_{L_{p,b}(H^1_2)} \nonumber\\
& + (\|\theta_0\|_{H^1_{2+\sigma}}+\|\pd_t\theta_2\|_{L_{p, b}(H^1_{2+\sigma})})
\|\nabla(\bv_1-\bv_2)\|_{L_{p,b}(H^1_2)} \nonumber\\
&
+(\|\pd_t(\theta_1-\theta_2)\|_{L_{p,b}(H^1_6)}
\|\nabla\bv_1\|_{L_{p,b}(H^1_{2+\sigma})}
+ \|\pd_t(\theta_1-\theta_2)\|_{L_{p,b}(H^1_{2+\sigma})}
\|\nabla\bv_1\|_{L_{p,b}(H^1_6)})\|\nabla\bv_1\|_{L_{p,b}(H^1_2)} \nonumber\\
&+\{(\|\theta_0\|_{H^1_6}+\|\pd_t\theta_2\|_{L_{p,b}(H^1_6)})
(\|\nabla(\bv_1-\bv_2)\|_{L_{p,b}(H^1_{2+\sigma})} 
+ \sum_{i=1,2}\|\nabla\bv_i\|_{L_{p,b}(H^1_{2+\sigma})}
\|\nabla(\bv_1-\bv_2)\|_{L_{p,b}(H^1_6))})
\nonumber\\
&+ (\|\theta_0\|_{H^1_{2+\sigma}}+\|\pd_t\theta\|_{L_{p,b}(H^1_{2+\sigma})})
(\|\nabla(\bv_1-\bv_2)\|_{L_{p,b}(H^1_6)} 
+ \sum_{i=1,2}\|\nabla\bv_i\|_{L_{p,b}(H^1_6)}
\|\nabla(\bv_1-\bv_2)\|_{L_{p,b}(H^1_6)})\} \nonumber \\
&\hskip10cm \times \|\nabla\bv_1\|_{L_{p,b}(H^1_2)} \nonumber \\
&+\{(\|\theta_0\|_{H^1_6}+\|\pd_t\theta_2\|_{L_{p,b}(H^1_6)})
\|\nabla\bv_2\|_{L_{p,b}(H^1_{2+\sigma})}
+ (\|\theta_0\|_{H^1_{2+\sigma}}+\|\pd_t\theta_2\|_{L_{p,b}(H^1_{2+\sigma})}
\|\nabla\bv_2\|_{L_{p,b}(H^1_6)}\} \nonumber \\
&\hskip10cm \times\|\nabla(\bv_1-\bv_2)\|_{L_{p,b}(H^1_2)}].
\label{nones:4.2}
\end{align}

We now estimate $\|F(\theta, \bv)\|_{L_{p,b}(H^1_q)}$ 
and $\|F(\theta_1, \bv_1) - F(\theta_2, \bv_2)\|_{L_{p, b}(H^1_q)}$
with $q=2$,  $2+\sigma$ and $6$.
For this purpose, we use the following estimates: 
\begin{align*}
\|fg\|_{L_{p, b}(H^1_q)} &\leq C\{\|f\|_{L_\infty(H^1_q)}\|g\|_{L_{p, b}(H^1_6)}
+ \|f\|_{L_\infty(H^1_q)}\|g\|_{L_{p,b}(H^1_6)}\}, \\
\|fgh\|_{L_{p, b}(H^1_q)} &\leq C\{\|f\|_{L_\infty(H^1_q)}
\|g\|_{L_\infty(H^1_6)}\|h\|_{L_{p, b}(H^1_6)}
+ \|f\|_{L_\infty(H^1_6)}\|g\|_{L_\infty(H^1_q)}\|h\|_{L_{p, b}(H^1_6)} \\
&\qquad + \|f\|_{L_\infty(H^1_6)}\|g\|_{L_\infty(H^1_6)}\|h\|_{L_{p, b}(H^1_q)}\}.
\end{align*}
And then, using \eqref{eq:4.0}, \eqref{eq:4.1}, \eqref{eq:4.2}, we have
\allowdisplaybreaks
\begin{align}
&\|F(\theta,\bv)\|_{L_{p,b}(H^1_q)}
 \leq C\{\|\nabla\bv\|_{L_{p,b}(H^1_q)}\|\nabla\bv\|_{L_{p,b}(H^1_6)}
+ (\|\theta_0\|_{H^1_q}+\|\pd_t\theta\|_{L_{p,b}(H^1_q)})\|\nabla\bv\|_{L_{p,b}(H^1_6)}
\nonumber \\
&\quad +  (\|\theta_0\|_{H^1_6}+\|\pd_t\theta\|_{L_{p,b}(H^1_6)})
\|\nabla\bv\|_{L_{p,b}(H^1_q)}
+  (\|\theta_0\|_{H^1_q}+\|\pd_t\theta\|_{L_{p,b}(H^1_q)})\|\nabla\bv\|_{L_{p,b}(H^1_6)}^2 
\nonumber \\
&\quad + (\|\theta_0\|_{H^1_6}+\|\pd_t\theta\|_{L_{p,b}(H^1_6)})\|\nabla\bv\|_{L_{p,b}(H^1_q)}
\|\nabla\bv\|_{L_{p,b}(H^1_6)}\};  \label{nones:4.3} \\
&\|F(\theta_1, \bv_1) - F(\theta_2, \bv_2)\|_{L_{p, b}(H^1_q)} \nonumber \\
&\quad \leq \{(\|\nabla(\bv_1-\bv_2)\|_{L_{p,b}(H^1_q)} + 
\sum_{i=1,2}\|\nabla\bv_i\|_{L_{p,b}(H^1_q)}\|\nabla(\bv_1-\bv_2)\|_{L_{p,b}(H^1_6)})
\|\nabla\bv_1\|_{L_{p,b}(H^1_6)} \nonumber \\
&\quad + (\|\nabla(\bv_1-\bv_2)\|_{L_{p,b}(H^1_6)} + 
\sum_{i=1,2}\|\nabla\bv_i\|_{L_{p,b}(H^1_6)}\|\nabla(\bv_1-\bv_2)\|_{L_{p,b}(H^1_6)})
\|\nabla\bv_1\|_{L_{p,b}(H^1_q)}\nonumber \\
&\quad +\|\nabla\bv_2\|_{L_{p,b}(H^1_q)}\|\nabla(\bv_1-\bv_2)\|_{L_{p,b}(H^1_6)}
+ \|\nabla\bv_2\|_{L_{p,b}(H^1_6)}\|\nabla(\bv_1-\bv_2)\|_{L_{p,b}(H^1_q)} \nonumber \\
&\quad + \|\pd_t(\theta_1-\theta_2)\|_{L_{p,b}(H^1_q)}\|\nabla\bv_1\|_{L_{p,b}(H^1_6)}
+ \|\pd_t(\theta_1-\theta_2)\|_{L_{p,b}(H^1_6)}
\|\nabla\bv_1\|_{L_{p,b}(H^1_q)} \nonumber \\ 
&\quad + (\|\theta_0\|_{H^1_q} + \|\pd_t\theta_2\|_{L_{p,b}(H^1_q)})
\|\nabla(\bv_1-\bv_2)\|_{L_{p,b}(H^1_6)}
+  (\|\theta_0\|_{H^1_6} + \|\pd_t\theta_2\|_{L_{p,b}(H^1_6)})
\|\nabla(\bv_1-\bv_2)\|_{L_{p,b}(H^1_q)}
\nonumber\\
&\quad +\|\pd_t(\theta_1-\theta_2)\|_{L_{p,b}(H^1_q)}\|\nabla\bv_1\|_{L_{p,b}(H^1_6)}^2 
+\|\pd_t(\theta_1-\theta_2)\|_{L_{p,b}(H^1_6)}\|\nabla\bv_1\|_{L_{p,b}(H^1_q)}
\|\nabla\bv_1\|_{L_{p,b}(H^1_6)} \nonumber \\
&\quad + (\|\theta_0\|_{H^1_q}+\|\pd_t\theta_2\|_{L_{p,b}(H^1_q)})(
\|\nabla(\bv_1-\bv_2)\|_{L_{p,b}(H^1_6)}
+ \sum_{i=1,2}\|\nabla\bv_1\|_{L_{p,b}(H^1_6)}\|\nabla(\bv_1-\bv_2)\|_{L_{p,b}(H^1_6)})
\nonumber \\
&\hskip13.2cm\times 
\|\nabla\bv_1\|_{L_{p,b}(H^1_6)}\nonumber \\
&\quad + (\|\theta_0\|_{H^1_6}+\|\pd_t\theta_2\|_{L_{p,b}(H^1_6)})(
\|\nabla(\bv_1-\bv_2)\|_{L_{p,b}(H^1_q)}
+ \sum_{i=1,2}\|\nabla\bv_1\|_{L_{p,b}(H^1_q)}\|\nabla(\bv_1-\bv_2)\|_{L_{p,b}(H^1_6)})
\nonumber \\
&\hskip13.2cm\times 
\|\nabla\bv_1\|_{L_{p,b}(H^1_6)}\nonumber \\
&\quad + (\|\theta_0\|_{H^1_6}+\|\pd_t\theta_2\|_{L_{p,b}(H^1_6)})(
\|\nabla(\bv_1-\bv_2)\|_{L_{p,b}(H^1_6)}
+ \sum_{i=1,2}\|\nabla\bv_1\|_{L_{p,b}(H^1_6)}\|\nabla(\bv_1-\bv_2)\|_{L_{p,b}(H^1_6)})
\nonumber \\
&\hskip13.2cm\times 
\|\nabla\bv_1\|_{L_{p,b}(H^1_q)}\nonumber \\
&\quad +(\|\theta_0\|_{H^1_q}+\|\pd_t\theta_2\|_{L_{p,b}(H^1_q)})
\|\nabla\bv_2\|_{L_{p,b}(H^1_6)}
\|\nabla(\bv_1-\bv_2)\|_{L_{p,b}(H^1_6)} \nonumber \\
&\quad +(\|\theta_0\|_{H^1_6}+\|\pd_t\theta_2\|_{L_{p,b}(H^1_6)})
\|\nabla\bv_2\|_{L_{p,b}(H^1_q)}
\|\nabla(\bv_1-\bv_2)\|_{L_{p,b}(H^1_6)} \nonumber \\
&\quad +(\|\theta_0\|_{H^1_6}+\|\pd_t\theta_2\|_{L_{p,b}(H^1_6)})
\|\nabla\bv_2\|_{L_{p,b}(H^1_6)}
\|\nabla(\bv_1-\bv_2)\|_{L_{p,b}(H^1_q)}.
\label{nones:4.4}
\end{align}

We next estimate $\|\bG(\theta, \bv)\|_{L_{p,b}(L_r)}$
and $\|\bG(\theta_1, \bv_1)-\bG(\theta_2, \bv_2)\|_{L_{p,b}(L_r)}$. For
this purpose, we use the estimates:
\begin{equation}\label{geneq:4.2}\begin{aligned}
\|fg\|_{L_{p,b}(L_r)} & \leq \|f\|_{L_\infty(L_{2+\sigma})}\|g\|_{L_{p,b}(L_2)}, \\
\|fgh\|_{L_{p,b}(L_r)} & \leq \|f\|_{L_\infty(L_\infty)}
\|g\|_{L_\infty(L_{2+\sigma)}}\|h\|_{L_{p,b}(L_2)}.
\end{aligned}\end{equation}
Employing the same argument as in \eqref{eq:4.0} and \eqref{eq:4.1}
and using $\bV_i(0)=0$ ($i=0,1$), for $i=0,1$ 
we have
\begin{equation}\label{eq:4.0.1}\begin{aligned}
&\|\bV_i(\bk)\|_{L_\infty(L_q)} \leq \sup_{|\bk| < \delta}|d\bV_i(\bk)|
\int^T_0\|\nabla\bv(\cdot, s)\|_{L_q} \leq C\|\nabla\bv\|_{L_{p,b}(L_q)}; \\
&\|\bV_i(\bk_{\bv_1}) - \bV_i(\bk_{\bv_2})\|_{L_\infty(L_q)}
\leq C\|\nabla(\bv_1-\bv_2)\|_{L_{p,b}(L_q)},
\end{aligned}\end{equation}
where $q=2, 2+\sigma$ and $6$. 
Moreover, $\|\bV_2(\bk)\|_{L_\infty(L_\infty)} = \sup_{|\bk| < \delta}|\bV_1(\bk)|$, 
\begin{align*}
&\|\bV_i(\bk)\|_{L_\infty(L_\infty)} \leq \sup_{|\bk| < \delta}|d\bV_i(\bk)|
\int^T_0\|\nabla\bv(\cdot, s)\|_{H^1_6} \leq C\|\nabla\bv\|_{L_{p,b}(H^1_6)};
\quad(i=0,1), \\
&\|\bV_i(\bk_{\bv_1}) - \bV_i(\bk_{\bv_2})\|_{L_\infty(L_\infty)}
\leq C\|\nabla(\bv_1-\bv_2)\|_{L_{p,b}(H^1_6)}
\quad(i=0,1,2) 
\end{align*}
as follows from  $|\bV_2(\bk_{\bv_1}) - \bV_2(\bk_{\bv_2})| 
\leq \sup_{|\bk| \leq \delta}|(d\bV_i)(\bk)||\bk_{\bv_1}-\bk_{\bv_2}|$. 
Writing 
\begin{align*}
\fp'(\rho_*+\theta)-\fp'(\rho_*)
&= \int^1_0\fp''(\rho_*
+\tau\theta)\,d\tau\,\theta, \\
\fp'(\rho_*+\theta_1)-\fp'(\rho_*+\theta_2) &
= \int^1_0\fp''(\rho_* + \theta_2
+\tau(\theta_1-\theta_2))\,d\tau\,(\theta_1-\theta_2),
\end{align*}
by  \eqref{cond:1} and \eqref{eq:4.2} we have
\begin{equation}\label{p-est.1}\begin{aligned}
&\|(\fp'(\rho_*+\theta)-\fp'(\rho_*))\nabla\theta\|_{L_{p,b}(L_r)}
\leq C(\|\theta_0\|_{L_{2+\sigma}}+\|\pd_t\theta\|_{L_{p,b}(L_{2+\sigma})}
\|\nabla\theta\|_{L_{p,b}(L_2)}, \\
&\|(\fp'(\rho_*+\theta_1)-\fp'(\rho_*))\nabla\theta_1
-(\fp'(\rho_*+\theta_2)-\fp'(\rho_*))\nabla\theta_2\|_{L_{p,b}(L_r)}
\\
&\quad 
\leq C\{\|\pd_t(\theta_1-\theta_2)\|_{L_{p,b}(L_{2+\sigma})}
\|\nabla\theta\|_{L_{p,b}(L_2)}
+ (\|\theta_0\|_{L_{2+\sigma}}+\|\pd_t\theta_2\|_{L_{p,b}(L_{2+\sigma})}
\|\nabla(\theta_1-\theta_2)\|_{L_{p, b}(L_2)}, \\
&\|(\fp'(\rho_*+\theta)-\fp'(\rho_*))\nabla\theta\|_{L_{p,b}(L_q)}
\leq C(\|\theta_0\|_{H^1_6}+\|\pd_t\theta\|_{L_{p,b}(H^1_6)}
\|\nabla\theta\|_{L_{p,b}(L_q)}, \\
&\|(\fp'(\rho_*+\theta_1)-\fp'(\rho_*))\nabla\theta_1
-(\fp'(\rho_*+\theta_2)-\fp'(\rho_*))\nabla\theta_2\|_{L_{p,b}(L_q)}
\\
&\quad 
\leq C\{\|\pd_t(\theta_1-\theta_2)\|_{L_{p,b}(H^1_6)}
\|\nabla\theta_1\|_{L_{p,b}(L_q)}
+ (\|\theta_0\|_{H^1_6}+\|\pd_t\theta_2\|_{L_{p,b}(H^1_6)}
\|\nabla(\theta_1-\theta_2)\|_{L_{p, b}(L_q)}, 
\end{aligned}\end{equation}
for $q=2, 2+\sigma$ and $6$. 
Combining these estimates above, we have 
\begin{align}
&\|\bG(\theta, \bv)\|_{L_{p, b}(L_r)}
\leq C\{(\|\theta_0\|_{L_{2+\sigma}}+\|\pd_t\theta\|_{L_{p,b}(L_{2+\sigma})})
(\|\pd_t\bv\|_{L_{p,b}(L_2)} + \|\nabla\theta\|_{L_{p,b}(L_2)}) \nonumber \\
&\quad + \|\nabla\bv\|_{L_{p,b}(L_{2+\sigma})}(\|\nabla^2\bv\|_{L_{p,b}(L_2)}
+ \|\nabla\theta\|_{L_{p,b}(L_2)})\};
\label{nones:4.5} \\
&\|\bG(\theta_1, \bv_1) - \bG(\theta_2, \bv_2)\|_{L_{p,b}(L_r)}
\leq C\{\|\pd_t(\theta_1-\theta_2)\|_{L_{p,b}(L_{2+\sigma})}\|\pd_t\bv_1\|_{L_{p,b}(L_2)}
\nonumber \\
&\quad + (\|\theta_0\|_{L_{2+\sigma}} + \|\pd_t\theta_2\|_{L_{p,b}(L_{2+\sigma})})
\|\pd_t(\bv_1-\bv_2)\|_{L_{p,b}(L_2)}
+ \|\nabla(\bv_1-\bv_2)\|_{L_{p,b}(L_2)}\|\nabla^2\bv_1\|_{L_{p, b}(L_{2+\sigma})}
\nonumber \\
&\quad + \|\nabla\bv_2\|_{L_{p,b}(L_{2+\sigma})}\|\nabla^2(\bv_1-\bv_2)\|_{L_{p,b}(L_2)}
+ \|\nabla(\bv_1-\bv_2)\|_{L_{p,b}(L_2)}\|\nabla^2\bv_1\|_{L_{p,b}(L_{2+\sigma})}
\|\nabla\bv_1\|_{L_{p,b}(H^1_6)} \nonumber \\
&\quad + \|\nabla^2(\bv_1-\bv_2)\|_{L_{p,b}(L_2)}\|\nabla\bv_1\|_{L_{p,b}(L_{2+\sigma})}
+ \|\nabla^2\bv_2\|_{L_{p,b}(L_{2+\sigma})}\|\nabla(\bv_1-\bv_2)\|_{L_{p,b}(L_2)}
\nonumber \\
&\quad +\|\pd_t(\theta_1-\theta_2)\|_{L_{p,b}(L_2)}
\|\nabla\theta_1\|_{L_{p,b}(L_{2+\sigma})}  
 + \|\nabla(\bv_1-\bv_2)\|_{L_{p,b}(L_2)}\|\nabla\theta_1\|_{L_{p,b}(L_{2+\sigma})} 
\nonumber \\
&\quad 
+ \|\nabla\bv_2\|_{L_{p,b}(L_{2+\sigma})}\|\nabla(\theta_1-\theta_2)\|_{L_{p,b}(L_2)}
+(\|\theta_0\|_{L_{2+\sigma}} + \|\pd_t\theta_2\|_{L_{p,b}(L_{2+\sigma})}
\|\nabla(\theta_1-\theta_2)\|_{L_{p,b}(L_2)}.  \label{nones:4.6}
\end{align}
Finally, we estimate $\|G(\theta, \bv)\|_{L_{p,b}(L_q)}$ and 
$\|G(\theta_1, \bv_1) - \bG(\theta_2, \bv_2)\|_{L_{p,b}(L_q)}$
with $q=2$,  $2+\sigma$, and $6$.
For this purpose, we use the following estimates: 
\begin{align*}
\|fg\|_{L_{p, b}(L_q)} &\leq C\|f\|_{L_\infty(H^1_q)}\|g\|_{L_{p, b}(L_q)}, \\
\|fgh\|_{L_{p, b}(H^1_q)} &\leq C\{\|f\|_{L_\infty(L_\infty)}
\|g\|_{L_\infty(H^1_6)}\|h\|_{L_{p, b}(L_q)}.
\end{align*}
And then, using \eqref{eq:4.0.1}, \eqref{p-est.1}, \eqref{eq:4.2} and \eqref{eq:4.3},
 for $q=2, 2+\sigma$ and $6$ we have
\begin{align}
&\|\bG(\theta, \bv)\|_{L_{p,b}(L_q)}
\leq C\{(\|\theta_0\|_{H^1_6} + \|\pd_t\theta\|_{L_{p,b}(H^1_6)})
(\|\pd_t\bv\|_{L_{p,b}(L_q)} + \|\nabla\theta\|_{L_{p,b}(L_q)}) \nonumber \\
&\quad +\|\nabla\bv\|_{L_{p,b}(H^1_6)}
(\|\nabla^2\bv\|_{L_{p,b}(L_q)} + \|\nabla\theta\|_{L_{p,b}(L_q)}); 
\label{nones:4.7}\\
&\|\bG(\theta_1, \bv_1) - \bG(\theta_2, \bv_2)\|_{L_{p,b}(L_q)} 
\leq C(\|\pd_t(\theta_1-\theta_2)\|_{L_{p,b}(H^1_6)}\|\pd_t\bv_1\|_{L_{p,b}(L_q)}
\nonumber \\
&\quad+(\|\theta_0\|_{H^1_6} + \|\pd_t\theta_2\|_{L_{p,b}(H^1_6)})
\|\pd_t(\bv_1-\bv_2)\|_{L_{p,b}(L_q)}
+ \|\nabla(\bv_1-\bv_2)\|_{L_{p,b}(H^1_6)}\|\nabla^2\bv_1\|_{L_{p,b}(L_q)}
\nonumber \\
&\quad + \|\nabla\bv_2\|_{L_{p,b}(H^1_6)}\|\nabla^2(\bv_1-\bv_2)\|_{L_{p,b}(L_q)}
+\|\nabla(\bv_1-\bv_2)\|_{L_{p,b}(H^1_6)}\|\nabla\bv_1\|_{L_{p,b}(H^1_6)}
\|\nabla^2\bv_1\|_{L_{p,b}(L_q)} \nonumber \\
&\quad + \|\nabla^2(\bv_1-\bv_2)\|_{L_{p,b}(L_q)}\|\nabla\bv_1\|_{L_{p,b}(H^1_6)}
+ \|\nabla^2\bv_2\|_{L_{p,b}(L_q)}\|\nabla(\bv_1-\bv_2)\|_{L_{p,b}(H^1_6)}
\nonumber \\
&\quad +\|\pd_t(\theta_1-\theta_2)\|_{L_{p,b}(H^1_6)}
\|\nabla\theta_1\|_{L_{p,b}(L_q)}  
 + \|\nabla(\bv_1-\bv_2)\|_{L_{p,b}(H^1_6)}\|\nabla\theta_1\|_{L_{p,b}(L_q)} 
\nonumber \\
&\quad 
+ \|\nabla\bv_2\|_{L_{p,b}(H^1_6)}\|\nabla(\theta_1-\theta_2)\|_{L_{p,b}(L_q)}
+(\|\theta_0\|_{H^1_6} + \|\pd_t\theta_2\|_{L_{p,b}(H^1_6)}
\|\nabla(\theta_1-\theta_2)\|_{L_{p,b}(L_q)}.
\label{nones:4.8}
\end{align}

\section{A priori estimates for solutions of linearized equations}

Let $\CV_{T, \epsilon} = \{(\theta, \bv) \in \CU^1_T\times \CU^2_T \mid 
E_T(\theta, \bv) \leq \epsilon\}$.  For $(\theta, \bv) \in \CV_{T, \epsilon}$, 
we consider linearized equations:
\begin{equation}\label{linearized:5.1}\begin{aligned}
\pd_t\eta + \rho_*\dv \bu = F(\theta, \bv)
&&\quad&\text{in $\Omega \times(0, T)$}, \\
\rho_*\pd_t\bu- \DV(\mu\bD(\bu) + \nu\dv\bu\bI - \fp'(\rho_*)\eta) 
= \bG(\theta, \bv)&&
\quad&\text{in $\Omega \times(0, T)$}, \\
\bu|_\Gamma=0, \quad (\eta, \bu)|_{t=0} = (\theta_0, \bv_0)&&
\quad&\text{in $\Omega$}. 
\end{aligned}\end{equation}
We first  show that equations \eqref{linearized:5.1} admit unique solutions
$\eta$ and $\bu$ with 
\begin{equation}\label{eq:5.4}\begin{aligned}
\eta &\in H^1_p((0, T), H^1_2(\Omega) \cap H^1_6(\Omega)), \\
\bu &\in H^1_p((0, T), L_2(\Omega)^3 \cap L_6(\Omega)^3)
\cap L_p((0, T), H^2_2(\Omega)^3\cap H^2_6(\Omega)^3)
\end{aligned}\end{equation}
possessing the estimate:
\begin{equation}\label{eq:5.5}
E_T(\eta, \bu) \leq C(\epsilon^2 + \epsilon^3)
\end{equation}
with some constant $C$ independent of $T$ and $\epsilon$. 

To prove \eqref{eq:5.5}, we divide $\eta$ and $\bu$ into two parts:
$\eta=\eta_1+ \eta_2$ and $\bu=\bu_1+\bu_2$, where $\eta_1$ and  $\bu_1$
are solutions of time shifted equations:
\begin{equation}\label{linearized:5.2}\begin{aligned}
\pd_t\eta_1 +\lambda_1\eta_1+ \rho_*\dv \bu_1= F(\theta, \bv)
&&\quad&\text{in $\Omega \times(0, T)$}, \\
\rho_*(\pd_t\bu_1+ \lambda\bu_1)
- \DV(\mu\bD(\bu_1) + \nu\dv\bu_1\bI - \fp'(\rho_*)\eta_1) = \bG(\theta, \bv)&&
\quad&\text{in $\Omega \times(0, T)$}, \\
\bu_1|_\Gamma=0, \quad (\eta_1, \bu_1)|_{t=0} = (\theta_0, \bv_0)&&
\quad&\text{in $\Omega$}, 
\end{aligned}\end{equation}
and $\eta_2$ and $\bu_2$ are solutions to  compensation equations:
\begin{equation}\label{linearized:5.3}\begin{aligned}
\pd_t\eta_2 + \rho_*\dv \bu_2= \lambda_1\eta_1
&&\quad&\text{in $\Omega \times(0, T)$}, \\
\rho_*\pd_t\bu_2 
- \DV(\mu\bD(\bu_2) + \nu\dv\bu_2\bI - \fp'(\rho_*)\eta_2) 
= \rho_*\lambda_1\bu_1&&
\quad&\text{in $\Omega \times(0, T)$}, \\
\bu_2|_\Gamma=0, \quad (\eta_2, \bu_2)|_{t=0} = (0, 0)&&
\quad&\text{in $\Omega$}. 
\end{aligned}\end{equation}
We first treat with equations \eqref{linearized:5.2}. For this purpose, we use the 
result stated in Sect. 3.  We consider a resolvent problem corresponding to 
equations \eqref{linearized:5.1} given as follows:
\begin{equation}\label{resolvent:5.1}\begin{aligned}
\lambda \zeta + \rho_*\dv \bw = f &&\quad&\text{in $\Omega$}, \\
\rho_*\lambda \bw- \DV(\mu\bD(\bw) + \nu\dv\bw\bI - \fp'(\rho_*)\zeta) = \bg&&
\quad&\text{in $\Omega$}, \\
\bw|_\Gamma=0&. 
\end{aligned}\end{equation}
Enomoto and Shibata \cite{ES1} proved the existence 
of $\CR$ bounded solution operators 
associated with \eqref{resolvent:5.1}.  Namely, we know the following theorem. 
\begin{thm} \label{thm:5.1}
Let $\Omega$ be a uniform $C^2$ domain in 
$\BR^N$. Let $0 < \omega < \pi/2$ and $1 < q < \infty$.  Set 
$H^{1,0}_q(\Omega) = H^1_q(\Omega)\times L_q(\Omega)^3$  
and $H^{1,2}_q(\Omega) = H^1_q(\Omega)\times H^2_q(\Omega)^3$. 
Then, there exist a large number $\lambda_0 > 0$ and 
operator families $\CP(\lambda)$ and $\CS(\lambda)$ with
$$\CP(\lambda) \in {\rm Hol}\,(\Sigma_{\omega, \lambda_0}, 
\CL(H^{1,0}_q(\Omega), H^1_q(\Omega))),
\quad
\CS(\lambda) \in {\rm Hol}\,(\Sigma_{\omega, \lambda_0}, 
\CL(H^{1,0}_q(\Omega), H^2_q(\Omega))$$
such that for any $\lambda \in \Sigma_{\omega, \lambda_0}$
and $(f, \bg) \in H^{1,0}_q(\Omega)$, $\zeta = \CP(\lambda)(f, \bg)$
and $\bw = \CS(\lambda)(f, \bg)$ are unique solutions of Stokes 
resolvent problem \eqref{resolvent:5.1} and 
\begin{align*}
\CR_{\CL(H^{1,0}_q(\Omega), H^1_q(\Omega))}(\{(\tau\pd_\tau)^\ell
(\lambda^k \CP(\lambda)) \mid \lambda \in \Sigma_{\omega, \lambda_0}\})
\leq r_b, \\
\CR_{\CL(H^{1,0}_q(\Omega), H^{2-j}_q(\Omega)^3)}(\{(\tau\pd_\tau)^\ell
(\lambda^{j/2}\CS(\lambda)) \mid \lambda \in \Sigma_{\omega, \lambda_0}\})
\leq r_b
\end{align*}
for $\ell=0,1$, $k=0,1$ and $j = 0,1,2$.  
\end{thm}
In view of Theorem \ref{thm:5.1} and consideration in Sect. 3, 
there exists a continuous analytic semigroup
$\{S(t)\}_{t\geq 0}$ associated with equations \eqref{linearized:5.2} such 
that
\begin{equation}\label{exp:5.1}
\|S(t)\|_{H^{1,0}_q(\Omega)} \leq C_qe^{-\lambda_2t}\|(f, \bg)\|_{H^{1,0}_q(\Omega)}
\end{equation}
for any $t> 0$ and $(f, \bg) \in H^{1,0}_q(\Omega)$ with some constant
$\lambda_2 > 0$.   Moreover, from
Theorem \ref{max:thm.2*} we have the following theorem.
\begin{thm}\label{thm:exp:5.1} Let $1<p, q < \infty$. 
Let  $b \geq 0$.  Then, there exists a large constant 
$\lambda_1 > 0$ such that for any $(f, \bg)$ with 
$<t>^b(f, \bg) \in L_p(\BR, H^{1,0}_q)$ and 
initial data $(\theta_0, \bv_0) \in H^1_q(\Omega)
\times B^{2(1-1/p)}_{q,p}(\Omega)^3$ satisfying the compatibility condition:
$\bv_0|_\Gamma=0$, problem:
\begin{equation}\label{linearized:5.4}\begin{aligned}
\pd_t\rho +\lambda_1\rho+ \rho_*\dv \bw= f
&&\quad&\text{in $\Omega \times(0, T)$}, \\
\rho_*(\pd_t\bw+ \lambda_1\bw)
- \DV(\mu\bD(\bw) + \nu\dv\bw\bI - \fp'(\rho_*)\rho) = \bg&&
\quad&\text{in $\Omega \times(0, T)$}, \\
\bw|_\Gamma=0, \quad (\rho, \bw)|_{t=0} = (\theta_0, \bv_0)&&
\quad&\text{in $\Omega$}, 
\end{aligned}\end{equation}
admits unique solutions $\rho\in H^1_p((0, T), H^1_q(\Omega))$ 
and $\bw\in H^1_p((0, T), L_q(\Omega)^3) \cap L_p((0, T), H^2_q(\Omega)^3)$ 
possessing the estimate:
\begin{align*}
&\|<t>^b(\rho, \pd_t\rho)\|_{L_p((0, T), H^1_q(\Omega))}
+ \|<t>^b\pd_t\bw\|_{L_p((0, T), L_q(\Omega))} 
+ \|<t>^b\bw\|_{L_p((0, T), H^2_q(\Omega))} \\
&\quad \leq C(\|\theta_0\|_{H^1_q(\Omega)} + \|\bv_0\|_{B^{2(1-1/p)}_{q, p}(\Omega)}
+ \|<t>^b(f, \bg)\|_{L_p((0, T), H^{1,0}_q(\Omega))}).
\end{align*}
Here, $C$ is a constant independent of $T>0$. 
\end{thm}
Applying Duhamel's principle to equations \eqref{linearized:5.2} yields that 
$$(\eta_1, \bu_1) = S(t)(\theta_0, \bv_0) + \int^t_0 S(t-s)
(F(\theta, \bv), \bG(\theta,\bv))(\cdot, s)\,ds.$$
Thus, by \eqref{exp:5.1}, we have
\begin{equation}\label{eq:5.6}\begin{aligned}
&\|<t>^b(\eta_1, \bu_1)\|_{L_p((0, T), H^{1,0}_r(\Omega))} \\
&\quad \leq C
(\|(\theta_0, \bv_0)\|_{H^{1,0}_r(\Omega)} + \|<t>^b(F(\theta, \bv), \bG(\theta, \bv))
\|_{L_p((0, T), H^{1,0}_r(\Omega))}). 
\end{aligned}\end{equation}
In fact, setting $I(t) = \int^t_0 S(t-s)(F(\theta, \bv), \bG(\theta, \bv))(\cdot, s)
\,ds$, by \eqref{exp:5.1} we have
\begin{align*}
<t>^b\|I(t)\|_{H^{1,0}_r(\Omega)} &\leq C_r
<t>^b\Bigl\{\int^{t/2}_0 + \int_{t/2}^t \Bigr\}e^{-\lambda_2(t-s)}
\|(F(\theta, \bv), \bG(\theta,\bv))(\cdot, s)\|_{H^{1,0}_r(\Omega)}\,ds \\
&= C_r (II(t)+ III(t)).
\end{align*}
In $II(t)$, using $e^{-\lambda_2(t-s)} \leq e^{-(\lambda_2/2)t}$ as follows from
$0 < s < t/2$,  by H\"older's inequality we have 
$$
II(t) \leq <t>^be^{-(\lambda_2/2)t}\Bigl(\int^\infty_0<s>^{-p'b}\,ds\Bigr)^{1/p'}
\Bigl(\int^T_0 (<s>^b\|(F(\theta, \bv), 
\bG(\theta,\bv))(\cdot, s)\|_{H^{1,0}_r(\Omega)})^p\,ds\Bigr)^{1/p},$$
and so we have
$$\Bigl(\int^T_0II(t)^p\,dt\Bigr)^{1/p}
\leq C\Bigl(\int^\infty_0(<t>^{b}e^{-(\lambda_2/2)t})^p\,dt\Bigr)^{1/p}
\|<t>^b(F(\theta, \bv), \bG(\theta, \bv))\|_{L_2((0, T), H^{1,0}_r(\Omega))}.
$$
On the other hand, using $<t>^b \leq C_b<s>^b$ for $t/2 < s < t$, 
by  H\"older's inequality we have
$$
III(t) \leq C_b\Bigl(\int^t_{t/2} e^{-\lambda_2(t-s)}\,ds\Bigr)^{1/p'}
\Bigl(\int^t_{t/2}e^{-\lambda_2(t-s)}(<s>^b
\|(F(\theta, \bv), \bG(\theta, \bv))(\cdot, s)\|_{L_r(\Omega)})^p\,
ds\Bigr)^{1/p}.
$$
Setting $L = \int^\infty_0 e^{-\lambda_2 t}\,dt$, by Fubini's theorem we have
$$\Bigl(\int^T_0III(t)^p\,dt\Bigr)^{1/p}
\leq C_b L \|<t>^b(F(\theta, \bv), \bG(\theta, \bv))\|_{L_p((0, T), H^{1,0}_r(\Omega))}.
$$
Combining these two estimates yields \eqref{eq:5.6}. 

Moreover, applying Theorem \ref{thm:exp:5.1} to equations \eqref{linearized:5.2}
yields that  
\begin{equation}\label{eq:5.7}\begin{aligned}
&\|<t>^b\pd_t(\eta_1, \bu_1)\|_{L_p((0, T), H^{1,0}_q(\Omega))}
+ \|<t>^b(\eta_1, \bu_1)\|_{L_p((0, T), H^{1,2}_q(\Omega))} \\
&\quad \leq C_q(\|\theta_0\|_{H^1_q(\Omega)} + \|\bv_0\|_{B^{2(1-1/p)}_{q,p}(\Omega)}
 + \|<t>^b(F(\theta, \bv), \bG(\theta, \bv))\|_{L_p((0, T), H^{1,0}_q(\Omega))})
\end{aligned}\end{equation}
for $q=2$, $2+\sigma$ and $6$. Recalling that 
$\|(\theta_0, \bv_0)\|_{\CI} \leq \epsilon^2$, by \eqref{nones:4.1}, 
\eqref{nones:4.3}, \eqref{nones:4.5}, \eqref{nones:4.7},
 \eqref{eq:5.6}, and \eqref{eq:5.7},  we have
\begin{equation}\label{eq:5.8}\begin{aligned}
&
\sum_{q=2, 2+\sigma, 6}(\|<t>^b\pd_t(\eta_1, \bu_1)\|_{L_p((0, T), H^{1,0}_q(\Omega))}
+ \|<t>^b(\eta_1, \bu_1)\|_{L_p((0, T), H^{1,2}_q(\Omega))} )
\\
&\quad + \|<t>^b(\eta_1, \bu_1)\|_{L_p((0, T), H^{1,0}_r(\Omega))}
\leq C(\epsilon^2 + \epsilon^3 + \epsilon^4). 
\end{aligned}\end{equation}
Here, $C$ is a constant independent of $T$ and $\epsilon$. 
By the trace method of real interpolation theorem, 
\begin{align*}
&\|<t>^b\bu_1\|_{L_\infty((0, T), L_q(\Omega))} \\
&\quad \leq C(\|\bv_0\|_{B^{2(1-1/p)}_{q,p}(\Omega)} + 
\|<t>^b\pd_t\bu_1\|_{L_p((0, T), L_q(\Omega))}
+ \|<t>^b\bu_1\|_{L_p((0, T), H^2_q(\Omega))}),
\end{align*}
and so by \eqref{eq:5.8} and $\|(\theta_0, \bv_0)\|_\CI \leq \epsilon^2$, 
\begin{equation}\label{add.est.1}
\sum_{q=2, 2+\sigma, 6}\|<t>^b\bu_1\|_{L_\infty((0, T), L_q(\Omega))}
\leq C(\epsilon^2 + \epsilon^3 + \epsilon^4).
\end{equation}

We now estimate $\eta_2$ and $\bu_2$.  Let $\{T(t)\}_{t\geq 0}$ be 
a continuous analytic semigroup associated with problem: 
\begin{equation}\label{modeleq:5.1}\begin{aligned}
\pd_t \rho + \rho_*\dv \bv = 0 &&\quad&\text{in $\Omega\times(0, \infty)$}, \\
\rho_*\pd_t \bv- \DV(\mu\bD(\bv) + \nu\dv\bv\bI - \fp'(\rho_*)\rho) = 0&&
\quad&\text{in $\Omega\times(0, \infty)$}, \\
\bv|_\Gamma=0\quad (\rho, \bv)|_{t=0}=(\theta_0, \bv_0)&&\quad
&\text{in $\Omega$}. 
\end{aligned}\end{equation}
By Theorem \ref{thm:5.1} and consideration in Sect. 3, we know
the existence of $C^0$ analytic semigroup $\{T(t)\}_{t\geq 0}$ associated with
\eqref{modeleq:5.1}. Moreover, by Enomoto and Shibata \cite{ES2}, 
we know that $\{T(t)\}_{t\geq 0}$ possesses the following $L_p$-$L_q$ 
decay estimates: 
Setting $(\theta, \bv) = T(t)(f, \bg)$, we have 
\begin{equation}\label{lp-lq}\begin{aligned}
\|(\theta, \bv)(\cdot, t)\|_{L_p} &\leq C_{p,q}t^{-\frac32\left(\frac1q-\frac1p\right)}
[(f, \bg)]_{p,q}\quad (t>1);\\
\|\nabla (\theta, \bv)(\cdot, t)\|_{L_p} &\leq C_{p,q}t^{-\sigma(p,q)}
[(f, \bg)]_{p,q} \quad(t>1);\\
\|\nabla^2\bv(\cdot, t)\|_{L_p} &\leq C_{p,q}t^{-\frac{3}{2q}}[(f, \bg)]_{p,q}
\quad(t > 1); \\
\|\pd_t (\theta, \bv)(\cdot, t)\|_{L_p}  &\leq Ct^{-\frac{3}{2q}}
[(f, \bg)]_{p,q}\quad(t > 1).
\end{aligned}\end{equation}
Here, $1 \leq q \leq 2 \leq p < \infty$, $[(f, \bg)]_{p,q}
= \|(f, \bg)\|_{H^{1,0}_p} + \|(f, \bg)\|_{L_q}$, $H^{m,n}_p = H^m_p\times H^n_p
\ni (f, \bg)$ and 
$$\sigma(p,q) = \frac32\left(\frac1q- \frac1p\right)+ \frac12
\quad(2 \leq p \leq 3), \quad \text{and}\quad 
\frac{3}{2q}\quad(p \geq 3).
$$
Moreover, we use 
\begin{equation}\label{lp-lq*}
\|(\theta, \bv)(\cdot, t)\|_{H^{1,2}_q} \leq M\|(f, \bg)\|_{H^{1,2}_q}
\quad(0 < t < 2)
\end{equation}
as follows from the following standard estimate for continuous analytic
semigroup. 
Applying Duhamel's principle to equations \eqref{linearized:5.3}
yields that
$$(\eta_2, \bu_2) = \lambda_1\int^t_0T(t-s)(\eta_1, \rho_*\bu_1)(\cdot, s)\,ds.
$$
Let 
$$
[[(\eta_1, \bu_1)(\cdot, s)]] = \|(\eta_1, \bu_1)(\cdot, s)\|_{H^{1,0}_r(\Omega)}
+ \sum_{q=2, 2+\sigma, 6}
\|(\eta_1, \bu_1)(\cdot, s)\|_{H^{1,2}_q(\Omega)}
+ \|\pd_t(\eta_1, \bu_1)(\cdots, s)\|_{H^{1,0}_q(\Omega)}). 
$$
We set 
$$\tilde E_T(\eta_1, \bu_1) : = \Bigl(\int^T_0(<t>^b[[\eta_1, \bu_1)(\cdot, t)]])^p\,
dt\Bigr)^{1/p},$$
and then, by \eqref{eq:5.8} we have
\begin{equation}\label{eq:5.9}
\tilde E_T(\eta_1, \bu_1) \leq C(\epsilon^2 + \epsilon^3 + \epsilon^4).
\end{equation}
First we consider the case: $2 \leq t  \leq T$. 
Notice  that 
\begin{align*}
&(1/2) + (3/2)(1/2 + 1/(2+\sigma) - 1/2) =
(3/2)(1/2+1/(2+\sigma) -1/6) \\
&\leq 
(1/2) + (3/2)(1/2 + 1/(2+\sigma) - 1/(2+\sigma))
\leq 3/(2r), 
\end{align*}
where $1/r = 1/2 + 1/(2+\sigma)$.  Let 
$\ell= (1/2) + (3/2)(1/2 + 1/(2+\sigma) - 1/2) = (5+\sigma)/(4+2\sigma)$, 
and then all the decay rates used below,  which are obtained by \eqref{lp-lq},
are less than or equal to $\ell$. 

Let $(\eta_3, \bu_3) = (\nabla\eta_2, \bar\nabla^1 \nabla\bu_2)$ 
when $q=2$ or $2+\sigma$, and $(\eta_3, \bu_3) = 
(\bar\nabla^1\eta_2, \bar\nabla^2\bu_2)$ when $q=6$. 
Here, $\bar\nabla^m f = (\pd_x^\alpha f \mid |\alpha| \leq m)$.
And then, 
\begin{align*}
&\|(\eta_3, \bu_3)(\cdot, t)\|_{L_q(\Omega)} \\
&\leq C\Bigl\{\int^{t/2}_0 + \int^{t-1}_{t/2} + \int^t_{t-1}\Bigr\}
\|(\nabla, \bar\nabla^1\nabla)\enskip\text{or}\enskip (\bar\nabla^1, \bar\nabla^2)
T(t-s)(\eta_1, \bu_1)(\cdot, s)\|_{L_q(\Omega)}\,ds \\
& = I_{q} + II_{q} + III_{q}.
\end{align*}

By \eqref{lp-lq},  we have 
\begin{align*}
I_{q}(t) &\leq C\int^{t/2}_0
(t-s)^{-\ell}
[[(\eta_1, \bu_1)]]\,ds \\
&\leq C(t/2)^{-\ell}\Bigl(\int^{t/2}_0<s>^{-b}<s>^b
[[(\eta_1, \bu_1)(\cdot, s)]]\,ds \\
& \leq Ct^{-\ell}
\Bigl(\int^T_0<s>^{-bp'}\,ds\Bigr)^{1/p'}\Bigl(\int^T_0
(<s>^b[[(\eta_1, \bu_1)(\cdot, s)]])^p\,ds\Bigr)^{1/p} \\
& \leq Ct^{-\ell}\tilde E_T(\eta_1, \bu_1).
\end{align*}
Recalling that $b = (3-\sigma)/(2(2+\sigma))$ when $p=2$ and 
$b=(1-\sigma)/(2(2+\sigma))$ when $p=1+\sigma$,  we see that 
$\ell- b = (2+2\sigma)/(2(2+\sigma)) > 1/2$ when $p=2$ and $
\ell-b = 1$ when $p=1+\sigma$. 
Thus, we have
$$\int^T_1(<t>^bI_{q}(t))^p\,dt 
\leq C\tilde E_T(\eta_1, \bu_1)^p.
$$
We next estimate $II_q(t)$. By \eqref{lp-lq} we have 
$$II_{q}(t) \leq C\int^{t-1}_{t/2}(t-s)^{-\ell}
[[(\eta_1, \bu_1)(\cdot, s)]]\,ds.$$
By H\"older's inequality and $<t>^b \leq C_b<s>^b$ for 
$s \in (t/2, t-1)$, we have 
\begin{align*}
<t>^bII_{q}(t) & \leq 
 C\int^{t-1}_{t/2}(t-s)^{-\ell/p'}(t-s)^{-\ell/p}<s>^b [[(\eta_1, \bu_1)(\cdot, s)]]\,
ds\\
&\leq C\Bigl(\int^{t-1}_{t/2}(t-s)^{-\ell}\,ds\Bigr)^{1/p'}
\Bigl(\int^{t-1}_{t/2}(t-s)^{-\ell}(<s>^b[[(\eta_1, \bu_1)(\cdot, s)]]^p)\,ds\Bigr)^{1/p}.
\end{align*}
Setting $\int^\infty_1 s^{-\ell}\,ds =L$, by Fubini's theorem we have 
\begin{align*}
\int^T_2(<t>^b II_{q}(t))^p\,dt
&\leq CL^{p/p'}\int^{T-1}_1(<s>^b[[(\eta_1, \bu_1)(\cdot, s)]])^p
\Bigl(\int^{2s}_{s+1} (t-s)^{-\ell}\,dt\Bigr) \,ds \\
& \leq CL^p\tilde E_T(\eta_1, \bu_1)^p.
\end{align*}
Using a standard estimate \eqref{lp-lq*} for continuous analytic semigroup, we have
\begin{align*}
III_{q}(t) &\leq C\int^t_{t-1}
\|(\eta_1, \bu_1)(\cdot, s)\|_{H^{1,2}_{q}}\,ds
\leq C\int^t_{t-1} [[(\eta_1, \bu_1)(\cdot, s)]]\,ds.
\end{align*}
Thus, employing the same argument as in estimating $II_{q}(t)$, 
we have
$$\int^T_2(<t>^b III_{q}(t))^p\,dt 
\leq C\tilde E_T(\eta_1, \bu_1)^p.
$$
Combining three estimates above yields that
\begin{equation}\label{eq:5.10}
\int^T_2(<t>^b\|(\eta_3, \bu_3)(\cdot, t)\|_{L_q(\Omega)})^p\,dt
\leq C\tilde E_T(\eta_1, \bu_1)^p,
\end{equation}
when $T > 2$.

For $0 < t < \min(2, T)$, using \eqref{lp-lq*}
and employing the same argument as in estimating
$III_q(t)$ above,   we have 
$$\int^{\min(2, T)}_0(<t>^b\|(\eta_3, \bu_3)(\cdot, t)\|_{L_q(\Omega)})^p\,dt
\leq C\tilde E_T(\eta_1, \bu_1)^p,
$$
which, combined with \eqref{eq:5.10},  yields that 
\begin{equation}\label{eq:5.11}
\int^T_0(<t>^b\|(\eta_3, \bu_3)(\cdot, t)\|_{L_q(\Omega)})^p\,dt
\leq C\tilde E_T(\eta_1, \bu_1)^p
\end{equation}
for $q=2$, $2+\sigma$, and $6$.

Since 
$$\pd_t(\eta_2, \bu_2) = -\lambda_1(\eta_1, \rho_*\bu_1)(\cdot, t)
-\lambda_1\int^t_0\pd_tT(t-s)(\eta_1, \rho_*\bu_1)(\cdot,s)\,ds, $$
employing the same argument as in proving \eqref{eq:5.11}, 
we have
\begin{equation}\label{eq:5.12}
\int^T_0(<t>^b\|\pd_t(\eta_2, \bu_2)(\cdot, t)\|_{L_q(\Omega)})^p\,dt
\leq C\tilde E_T(\eta_1, \bu_1)^p
\end{equation}
for $q=2$, $2+\sigma$, and $6$. 

We now estimate $\sup_{2 < t < T}<t>^b \|(\eta_2, \bu_2)\|_{L_q(\Omega)}$
for $q=2, 2+\sigma$ and $6$.
Let $q=2$, $2+\sigma$ and $6$ in what follows. For $2 < t < T$, 
\begin{align*}
\|(\eta_2, \bu_2)(\cdot, t)\|_{L_q(\Omega)} 
&\leq C\Bigl\{\int^{t/2}_0 + \int^{t-1}_{t/2} + \int^t_{t-1}\Bigr\}
\|T(t-s)(\eta_1, \bu_1)(\cdot, s)\|_{L_q(\Omega)}\,ds \\
& = I_{q, 0} + II_{q, 0} + III_{q, 0}.
\end{align*}
By \eqref{lp-lq}, we have
\begin{align*}
I_{q,0}(t) &\leq C\int^{t/2}_0(t-s)^{-3/2(2+\sigma)}[[(\eta_1, \bu_1)(\cdot, s)]]\,ds
\\
&\leq C(t/2)^{-3/2(2+\sigma)}\int^{t/2}_0<s>^{-b}<s>^b[[(\eta_1, \bu_1)(\cdot, s)]]\,ds
\\
&\leq Ct^{-3/2(2+\sigma)}\Bigl(\int^\infty_0<s>^{-p'b}\,ds\Bigr)^{1/p'}
\tilde E_T(\eta_1, \bu_1).
\end{align*}
Note that $3/2(2+\sigma) = (3/2)(1/r-1/2) < (3/2)(1/r-1/(2+\sigma))<(3/2)(1/r-1/6)$.
By \eqref{lp-lq}, we also have
\begin{align*}
II_{q,0}(t) & \leq C\int^{t-1}_{t/2}(t-s)^{-3/2(2+\sigma)}\|(\eta_1, \bu_1)(\cdot, s)]]\,ds
\\
& \leq C\Bigl(\int^{t-1}_{t/2}((t-s)^{-3/2(2+\sigma)}<s>^{-b})^{p'}\,ds\Bigr)^{1/p'}
\Bigl(\int^{t-1}_{t/2}(<s>^b[[(\eta_1, \bu_1)(\cdot, s)]])^p\,ds\Bigr)^{1/p}\\
&\leq C<t>^{-b}\tilde E_T(\eta_1, \bu_1),
\end{align*}
where we have used $3p'/2(2+\sigma) > 1$. By \eqref{lp-lq*}, we have
\begin{align*}III_{q,0}(t) & \leq C\int^t_{t-1}[[(\eta_1, \bu_1)(\cdot, s)]]\,ds\\
& \leq C<t>^{-b}\int^t_{t-1}<s>^b[[(\eta_1, \bu_1)(\cdot, s)]]\,ds \\
&\leq C<t>^{-b}\Bigl(\int^t_{t-1}\,ds\Bigr)^{1/p'}\tilde E_T(\eta_1, \bu_1).
\end{align*}
Since $b< 3/2(2+\sigma)$, 
combining these estimates above yields that 
\begin{equation}\label{eq:5.13}
\sup_{2 < t < T}<t>^b\|(\eta_1, \bu_1)(\cdot, t)\|_{L_q(\Omega)} 
\leq C\tilde E_T(\eta_2, \bu_2)
\end{equation}
 For $0 < t < \min(2, T)$, by standard estimate
\eqref{lp-lq*} of continuous analytic semigroup, we have
$$
\sup_{0 < t < \min(2, T)}<t>^b\|(\eta_1, \bu_1)(\cdot, t)\|_{L_q(\Omega)} 
\leq C\tilde E_T(\eta_2, \bu_2)
$$
which, combined with \eqref{eq:5.13}, yields that 
\begin{equation}\label{eq:5.14*}
\|<t>^b(\eta_1, \bu_1)(\cdot, t)\|_{L_\infty((0, T), L_q(\Omega)} 
\leq C\tilde E_T(\eta_2, \bu_2)
\end{equation}
for $q=2, 2+\sigma$ and $6$. 

Recalling that 
$\eta=\eta_1+\eta_2$ and $\bu=\bu_1+\bu_2$, 
noting that
$E_T(\eta_1, \bu_1) \leq C(\tilde E_T(\eta_1, \bu_1)+\|(\theta_0, \bv_0)\|_\CI)$ 
as follows from \eqref{add.est.1}, and combining
\eqref{eq:5.11}, \eqref{eq:5.12}, \eqref{eq:5.14*}, and \eqref{eq:5.9} yield that
\begin{equation}\label{eq:5.14}
E_T(\eta, \bu) \leq C(\epsilon^2 + \epsilon^3 + \epsilon^4).
\end{equation}
If we choose $\epsilon > 0$ so small that
$C(\epsilon + \epsilon^2 + \epsilon^3) < 1$ in \eqref{eq:5.14}, we have
$E_T(\eta, \bu) \leq \epsilon$.  Moreover, by \eqref{eq:4.3} 
$$\sup_{t \in (0, T)} \|\eta(\cdot, t)\|_{L_\infty(\Omega)}
\leq C(\|\eta_0\|_{H^1_6} + \|\pd_t\eta\|_{L_p((0, T), H^1_6(\Omega))})
\leq C(\epsilon^2 +\epsilon^3 + \epsilon^4).$$
Thus, choosing $\epsilon > 0$ so small that 
$C(\epsilon^2 +\epsilon^3 + \epsilon^4) \leq \rho_*/2$,
 we see that 
$\sup_{t \in (0, T)} \|\eta(\cdot, t)\|_{L_\infty(\Omega)} \leq \rho_*/2$. 
And also, 
$$\int^T_0\|\nabla\bu(\cdot, s)\|_{L_\infty(\Omega)}\,ds 
\leq \Bigl(\int^\infty_0<s>^{-p'b}\,ds\Bigr)^{1/p'}
\|<t>^b\nabla\bu\|_{L_p((0, T), H^1_6(\Omega))}
\leq C_{p', b}(\epsilon^2 + \epsilon^3 + \epsilon^4).
$$
Thus, choosing $\epsilon > 0$ so small that 
$C_{p', b}(\epsilon^2 + \epsilon^3 + \epsilon^4) \leq \delta$,
 we see that 
$\int^T_0\|\nabla\bu(\cdot, s)\|_{L_\infty(\Omega)}\,ds \leq \delta$.
From consideration above, it follows that $(\eta, \bu) 
\in \CV_{T, \epsilon}$. Let $\CS$ be an operator defined by 
$\CS(\theta, \bv) = (\eta, \bu)$ for $(\theta, \bv)
\in \CV_{T, \epsilon}$, and then $\CS$ maps $\CV_{T, \epsilon}$
into itself.  

We now show that $\CS$ is a contraction map.  Let 
$(\theta_i, \bv_i) \in \CV_{T, \epsilon}$ ($i=1,2$) and set 
$(\eta, \bu) = (\eta_1, \bu_1) - (\eta_2, \bu_2) = \CS(\theta_1, \bv_1)-
\CS(\theta_2, \bv_2)$, and $F = F(\theta_1, \bv_1)-F(\theta_2, \bv_2)$ and 
$\bG = \bG(\theta_1, \bv_1) - \bG(\theta_2, \bv_2)$.  And then, 
from \eqref{linearized:5.1} it follows that 
\begin{equation}\label{difeq:5.1}\begin{aligned}
\pd_t\eta + \rho_*\dv \bu = F
&&\quad&\text{in $\Omega \times(0, T)$}, \\
\rho_*\pd_t\bu- \DV(\mu\bD(\bu) + \nu\dv\bu\bI - \fp'(\rho_*)\eta) 
= \bG&&
\quad&\text{in $\Omega \times(0, T)$}, \\
\bu|_\Gamma=0, \quad (\eta, \bu)|_{t=0} = (0, 0)&&
\quad&\text{in $\Omega$}. 
\end{aligned}\end{equation}
By \eqref{nones:4.2}, \eqref{nones:4.4}, \eqref{nones:4.6}, and \eqref{nones:4.8}, 
we have
$$\|(F, \bG)\|_{L_p((0, T), H^{1,0}_r(\Omega))} 
+ \sum_{q=2, 2+\sigma, 6}\|(F, \bG)\|_{L_p((0, T), H^{1,0}_q(\Omega))}
\leq C(\epsilon + \epsilon^2+ \epsilon^3)E_T((\theta_1, \bv_1)-(\theta_2, \bv_2)).
$$
Applying the same argument as in proving \eqref{eq:5.14} to equations
\eqref{difeq:5.1} and recalling $(\eta, \bu) = \CS(\theta_1, \bv_1)-
S(\theta_2, \bv_2)$,
we have
$$E_T( \CS(\theta_1, \bv_1)-
S(\theta_2, \bv_2)) 
\leq C(\epsilon + \epsilon^2+ \epsilon^3)E_T((\theta_1, \bv_1)-(\theta_2, \bv_2)),
$$
for some constant $C$ independent of $\epsilon$ and $T$.  Thus, choosing
$\epsilon > 0$ so small that $C(\epsilon + \epsilon^2+ \epsilon^3) < 1$, we have
that $\CS$ is a contraction map on $\CV_{T, \epsilon}$, which proves  
Theorem \ref{mainthm:2}. Since the contraction mapping principle yields
the uniqueness of solutions in $\CV_{T, \epsilon}$, 
we have completed the proof of Theorem \ref{mainthm:2}. 
\section{A proof of Theorem \ref{thm:main0}}

We shall prove Theorem \ref{thm:main0} with the help of Theorem \ref{mainthm:2}.
In what follows, let $b$ and $p$ be the constants given in Theorem \ref{mainthm:2},
and $q=2, 2+\sigma$ and $6$. 
As was stated in Sect. 2, the Lagrange transform \eqref{lag:1} gives a
$C^{1+\omega}$ ($\omega \in (0, 1/2)$) diffeomorphism on $\Omega$ and 
$dx= \det(\bI + \bk)\,dy$, where $\{x\}$ and $\{y\}$
denote respective Euler coordinates and Lagrange coordinates on $\Omega$
and $\bk = \int^t_0\nabla \bu(\cdot, s)\,ds$. 
By \eqref{lag:2}, $\|\bk\|_{L_\infty(\Omega)} \leq \delta < 1$. In particular, 
choosing $\delta>0$ smaller if necessary, we may assume that 
$C^{-1}\leq \det(\bI + \int^t_0\nabla\bu(\cdot, s) \,ds)\leq C$ with some constant
$C > 0$ for any 
$(x, t) \in \Omega\times(0, T)$. Let $y = X_t(x)$ be an inverse map of Lagrange transform \eqref{lag:1}, and set $\theta(x, t) = \eta(X_t(x), t)$ and $\bv(x, t)
= \bu(X_t(x), t)$. We have 
$$\|(\theta, \bv)\|_{L_q(\Omega)} \leq C\|(\eta, \bu)\|_{L_q(\Omega)}.$$
Noting that $(\eta, \bu)(y, t) = (\theta, \bv)(y+\int^t_0\bu(y, s)\,ds, t)$, 
the chain rule of composite functions yields that 
\begin{align*}
\|(\nabla(\theta, \bv)\|_{L_q(\Omega)} \leq C
(1-\|\bk\|_{L_\infty(\Omega)})^{-1}\|\nabla(\eta, \bu)\|_{L_q(\Omega)}; \\
\|\nabla^2\bv\|_{L_q(\Omega)} \leq C(1-\|\bk\|_{L_\infty(\Omega)})^{-2}
\|\nabla^2\bu\|_{L_q(\Omega)} + (1-\|\bk\|_{L_\infty(\Omega)})^{-1}
\|\nabla\bk\|_{L_q(\Omega)}\|\nabla\bu\|_{L_\infty(\Omega)}.
\end{align*} 
Thus, using 
$\|\nabla\bk\|_{L_q(\Omega)} \leq C\|<t>^b\nabla^2\bu\|_{L_p((0, T), L_q(\Omega))}$
and $\|\nabla\bu\|_{L_\infty(\Omega)}\leq C\|\nabla \bu\|_{H^1_6(\Omega)}$, we have
\begin{align*}
\|<t>^b\nabla(\theta, \bv)\|_{L_\infty((0, T), L_2(\Omega) \cap L_6(\Omega))}
&\leq C\|<t>^b\nabla(\theta, \bv)\|_{L_\infty((0, T), L_2(\Omega) \cap L_6(\Omega))};
\\
\|<t>^b(\theta, \bv)\|_{L_p((0, T), L_6(\Omega))}
&\leq C\|<t>^b(\theta, \bv)\|_{L_p((0, T),  L_6(\Omega))};
\\
\|<t>^b(\theta, \bv)\|_{L_\infty((0, T), L_2(\Omega) \cap L_6(\Omega))}
&\leq C\|<t>^b(\theta, \bv)\|_{L_p((0, T), L_2(\Omega) \cap L_6(\Omega))};
\\
\|<t>^b\nabla^2\bv\|_{L_p((0, T), L_2(\Omega) \cap L_6(\Omega))}
&\leq C(\|<t>^b\nabla^2\bu\|_{L_p((0, T), L_2(\Omega) \cap L_6(\Omega))}\\
&+ \|<t>^b\nabla^2\bu\|_{L_p((0, T), L_q(\Omega))}\|<t>^b\nabla \bu
\|_{L_p((0, T), H^1_6(\Omega))}).
\end{align*}
Since
$\pd_t(\eta, \bu)(y, t) = \pd_t[(\theta,\bv)(y+ \int^t_0\bu(y, s)\,ds, t)]
= \pd_t(\theta, \bv)(x, t) +\bu\cdot \nabla(\theta, \bv)(x, t)$, 
we have 
\begin{align*}
\|\pd_t(\theta, \bv)\|_{L_q(\Omega)} 
\leq C\|\pd_t(\eta, \bu)\|_{L_q(\Omega)} + \|\bu\|_{L_\infty(\Omega)}
\|\nabla\eta\|_{L_q(\Omega)} 
+ \|\bu\|_{L_q(\Omega)}\|\nabla\bu\|_{L_\infty(\Omega)}.
\end{align*}
Since
$\|\nabla\eta\|_{L_\infty((0, T), L_q(\Omega))}
\leq \|\nabla \theta_0\|_{L_q(\Omega)} 
+ C\|<t>^b\pd_t\eta\|_{L_p((0, T), H^1_q(\Omega))}$, we have
\begin{align*}
\|<t>^b\pd_t(\theta, \bv)\|_{L_p((0, T), L_q(\Omega))}
&\leq C(\|<t>^b\pd_t(\eta, \bu)\|_{L_p((0, T), L_q(\Omega))} \\
&+ (\|\nabla \theta_0\|_{L_q(\Omega)} 
+ \|<t>^b\pd_t\eta\|_{L_p((0, T), H^1_q(\Omega))})\|<t>^b
\bu\|_{L_p((0, T), H^1_6(\Omega))}\\
&+ \|<t>^b\bu\|_{L_\infty((0, T), L_q(\Omega))}
\|<t>^b\nabla\bu\|_{L_p((0, T), H^1_6(\Omega))}).
\end{align*}
By Theorem \ref{mainthm:2} we see that there exists a small constant
$\epsilon > 0$ such that if initial data $(\theta_0,\bv_0) \in \CI$ 
satisifes the compatibility condition: $\bv_0|_\Gamma=0$ and 
the smallness condition: $\|(\theta_0, \bv_0)\|_{\CI} \leq \epsilon^2$
then problem \eqref{eq:1.1} admits unique solutions 
$\rho = \rho_*+\theta$ and $\bv$ satisfying the regularity conditions
\eqref{sol:1} and $\CE(\theta, \bv) \leq \epsilon$. This completes the proof
of Theorem \ref{thm:main0}. 
\section{Comment on the proof}
Let $N \geq 3$ and $\Omega$ be an exterior domain in 
$\BR^N$. 
Assume that  $L_p$-$L_q$ decay estmates for $C_0$ analytic semigroup
like \eqref{lp-lq} are valid.  
We choose $q_1=2$, $q_2 = 2+\sigma$,  and $q_3$ in such a way that $q_3 > N$ and 
$$\frac12 + \frac{N}{2(2+\sigma)} \leq \frac{N}{2}\Bigl(\frac12+\frac{1}{2+\sigma}
-\frac{1}{q_3}\Bigr). $$
Namely, $q_3 =6$ ($N=3$) and $q_3 >N \geq  2N/(N-2)$ for $N \geq 4$. 
If  $L_1$ in space estimates hold, then the global well-posedness is 
established with $q_1=q_2=2$.  But, so far
 $L_1$ in space estimates does not hold, and so we have chosen
$q_1=2$ and $q_2 = 2+\sigma$. 
Let 
$p$ and $b$ be chosen in such a way that 
$$\Bigl(\frac12 + \frac{N}{2(2+\sigma)} - b\Bigr)p > 1, 
\quad bp' > 1.$$
If we write equations as 
$$\pd_tu - Au = f, \quad Bu = g \quad(t > 0), \quad u|_{t=0}=u_0.$$
Here, $Bu=g$ is corresponding to boundary conditions, 
and $f$ and $g$ are corresponding to nonlinear terms. The first reduction
is that $u_1$ is a solution to equations: 
$$\pd_tu_1 + \lambda_1 u_1 - Au_1 = f, 
\quad Bu_1 = g \quad(t > 0), \quad u_1|_{t=0}=u_0.$$
Then, $u_1$ has the same decay properties as nonlinear terms $f$ and $g$ have. 
If $u_1$ does not belong to the domain of the operator $(A, B)$ 
(free boundary conditions or slip boundary conditions cases)), 
in addition we choose $u_2$ as 
 a solution of equations:
$$\pd_tu_2 + \lambda_1 u_2 - Au_2 = \lambda_1u_1, \quad Bu_2 = 0
 \quad(t > 0), \quad u_2|_{t=0}=0,$$
with very large constant $\lambda_1 > 0$. 
Since $u_2$ belongs to the domain of operator $A$ for any $t > 0$, 
we choose $u_3$
 as a solution of equations:
$$\pd_tu_3 - Au_3 = \lambda_1 u_2, \quad Bu_3 = 0 \quad(t > 0), \quad u_3|_{t=0}=0.$$
And then, by the Duhamel principle, we have
$$u_3= \lambda_1\int^t_0T(t-s)u_2(s)\,ds,
$$
and we use \eqref{lp-lq}
estimate for $0 < s < t-1$  and 
a standard semigroup estimate for $t-1 < s < t$, that is
$\|T(t-s)u_2(s)\|_{D(A)} \leq C\|u(s)\|_{D(A)}$ for $t-1<s<t$, where $\|\cdot\|_{D(A)}$
is a domain norm. 

When $N=2$, the method above is fail, because 
$$\frac12 + \frac{2}{2(2+\sigma)} < 1.$$
And so, Matsumura-Nishida method seems to be only the way to prove the 
global wellposedness in a two dimensiona exterior domain. 

\end{document}